\documentclass[11pt,a4paper]{article}
\usepackage{pict2e}
\usepackage{booktabs}
\usepackage{color}
\usepackage{epsfig}
\usepackage{epstopdf}
\usepackage{longtable}
\usepackage[T1]{fontenc}
\usepackage{amsmath}
\usepackage{cases}
\usepackage{geometry}
\usepackage{amsbsy,latexsym,amsfonts, epsfig, color, authblk, amssymb, graphics, bm}
\usepackage{epsf,slidesec,epic,eepic}
\usepackage{fancybox}
\usepackage{fancyhdr}
\usepackage{setspace}
\usepackage{nccmath}
\usepackage{diagbox}
\usepackage[colorlinks, citecolor=blue]{hyperref}

\newtheorem{theorem}{Theorem}[section]
\newtheorem{lemma}{Lemma}[section]
\newtheorem{algorithm}{Algorithm}[section]

\newtheorem{example}{Example}[section]
\newtheorem{proposition}{Proposition}[section]

\newtheorem{assumption}{Assumption}%[section]
\newtheorem{remark}{Remark}[section]

\newtheorem{alemma}{Lemma}

\newenvironment{proof}{{\noindent \bf Proof:}}{\hfill$\Box$\medskip}
\definecolor{lred}{rgb}{1,0.8,0.8}
\definecolor{lblue}{rgb}{0.8,0.8,1}
\definecolor{dred}{rgb}{0.6,0,0}
\definecolor{dblue}{rgb}{0,0,0.5}
\definecolor{dgreen}{rgb}{0,0.5,0.5}

\definecolor{blue}{rgb}{0,0,0.9}
\definecolor{red}{rgb}{0.9,0,0}
\definecolor{green}{rgb}{0,0.9,0}

 \title{A multi-stage convex relaxation approach to noisy structured low-rank matrix recovery}

 \author{Shujun Bi\footnote{School of Mathematics, South China University of Technology, Tianhe District of Guangzhou City, China (bishj@scut.edu.cn).},
 \ \ Shaohua Pan\footnote{Corresponding author. School of Mathematics, South China University of Technology, Tianhe District of Guangzhou City, China (shhpan@scut.edu.cn).}\ \ and \  Defeng Sun\footnote{Department of Mathematics and Risk Management Institute, National University of Singapore, 10 Lower Kent Ridge Road, Singapore 119076 (matsundf@nus.edu.sg). }}

  \date{March 1, 2017}

\begin{document}

  \maketitle

 \begin{abstract}
  This paper concerns with a noisy structured low-rank matrix recovery problem which can be modeled
  as a structured rank minimization problem. We reformulate this problem as
  a mathematical program with a generalized complementarity constraint \!(MPGCC),
  and show that its penalty version, yielded by moving the generalized complementarity constraint
  to the objective, has the same global optimal solution set as the MPGCC does
  whenever the penalty parameter is over a threshold. Then, by solving the exact penalty
  problem in an alternating way, we obtain a multi-stage convex relaxation approach.
  We provide   theoretical guarantees for our approach under a mild restricted eigenvalue condition,
  by quantifying the reduction of the error and approximate rank bounds of the first stage convex
  relaxation (which is exactly the nuclear norm relaxation) in the subsequent stages
  and establishing the geometric convergence of the error sequence in a statistical sense.
  Numerical experiments are conducted for some structured low-rank matrix recovery examples  to
  confirm our theoretical findings.
 \end{abstract}



%------------------------------------------------------------------------------------------------Section 1
  \section{Introduction}\label{sec1}

  The task of noisy structured low-rank matrix recovery is to find a low-rank matrix
  with a certain structure consistent with some noisy linear measurements.
  Let $\overline{X}$ be the target matrix to be recovered and $b=\mathcal{A}\overline{X}+\xi$
  be the noisy measurement vector, where $\mathcal{A}\!:\mathbb{R}^{n_1\times n_2}\to\mathbb{R}^m$
  is the sampling operator and $\xi\in\mathbb{R}^m$ is the noisy vector with $\|\xi\|\le \delta$
  for some $\delta>0$. The noisy structured low-rank matrix recovery problem can be modeled as
  \begin{equation}\label{rank-min}
    \min_{X\in\mathbb{R}^{n_1\times n_2}}\!\big\{{\rm rank}(X)\!:\ \|\mathcal{A}X-b\|\le \delta,\, X\in\Omega\big\},
  \end{equation}
  where $\Omega\subseteq\mathbb{R}^{n_1\times n_2}$ is a compact convex set to represent
  the structure of $\overline{X}$. Without loss of generality, we  assume that
  $\mathcal{A}X\!:=(\langle A_1,X\rangle,\ldots,\langle A_m,X\rangle)^{\mathbb{T}}$
  for $X\in\mathbb{R}^{n_1\times n_2}$, where $A_1,\ldots,A_m$ are the given matrices
  in $\mathbb{R}^{n_1\times n_2}$. Such a structured rank minimization problem has wide
  applications in system identification and control \cite{Fazel02,FPST13},
  signal and image processing \cite{Haeffele14,ChenChi14}, machine learning \cite{Richard12},
  multi-dimensional scaling in statistics \cite{QiY14}, finance \cite{Pietersz04}, quantum tomography
  \cite{Gross11}, and so on. For instance, one is often led to seek a low-rank Hankel matrix
  in system identification and control, a low-rank correlation matrix in finance
  and a low-rank density matrix in quantum tomography.

  \medskip

  Due to  the combinatorial property of the rank function, problem \eqref{rank-min} is generally NP-hard.
  One popular way to  deal with NP-hard problems is to use the convex relaxation technique,
  which typically yields a desirable local optimal or at least a feasible solution via solving a single
  or a sequence of numerically tractable convex optimization problems. Fazel \cite{Fazel02}
  initiated the research for the nuclear norm relaxation method, motivated by the fact
  that the nuclear norm is the convex envelope of the rank function in the unit ball on the spectral norm.
  In the past ten years, this relaxation method has received much attention from many fields such as
  information, computer science, statistics, optimization, and so on (see, e.g.,
  \cite{Candes09,Gross11,Recht10,KesMO10,KolLT11,Negahban11,Toh10}), and it has been shown that
  a single nuclear norm minimization problem can recover the target matrix
  $\overline{X}$ under a certain restricted isometry property (RIP) of $\mathcal{A}$ when $\delta\!=0$
  \cite{Recht10} or can yield a solution satisfying a certain error bound when $\delta>0$
  \cite{Candes11}. For its recoverability and error bounds under other conditions,
  the interested readers may refer to the literature
  \cite{Dvijotham-SSP,Negahban11,Recht-NSP} and references therein.

  \medskip

  Most of the existing low-rank matrix optimization models are focused on the case that
  $\Omega=\mathbb{R}^{n_1\times n_2}$. When the structure on the target matrix is known,
  it is reasonable to consider the rank minimization problem \eqref{rank-min} with $\Omega$
  indicating the available information. However, the (hard) constraint $X\in\Omega$ often contradicts
  the role of the nuclear norm in promoting a low-rank solution. For example, when $\Omega$ consists of
  the set of correlation matrices, the nuclear norm relaxation method for \eqref{rank-min} may fail in
  generating a low-rank solution since the nuclear norm becomes a constant in the set $\Omega$.
  In addition, although some error bounds have been established for the nuclear norm relaxation method
  in the noisy setting \cite{Candes11,Negahban11,Negahban12}, they are minimax-optimal up to a logarithmic
  factor of the dimension \cite{Negahban12}, instead of a constant factor like the $l_1$-norm relaxation method
  for sparse regression \cite{Raskutti11}. These two considerations motivate
  us to seek more efficient convex relaxations.

%-------------------------------------------------------------------------------------------------------------Subsection
  \subsection{Our main contribution}\label{subsec1.1}

  The main contribution of this work is the introduction of a multi-stage convex relaxation approach via
  an equivalent Lipschitz optimization reformulation. This approach can efficiently reduce
  the error bounds obtained from the nuclear norm convex relaxation. More specifically,
  we reformulate problem \eqref{rank-min} as an equivalent MPGCC by using a variational
  characterization of the rank function and verify that its penalized version,
  yielded by moving the generalized complementarity constraint to the objective,
  has the same global optimal solution set as the MPGCC does once the penalty parameter
  is over a threshold. This exact penalty problem not only has a convex feasible set
  but also possesses a Lipschitz objective function with a bilinear structure, which offers
  a favorable Lipschitz reformulation for problem \eqref{rank-min}. To the best of our knowledge,
  this is the first equivalent Lipschitz characterization for low-rank matrix optimization problems.
  Then with this reformulation, we propose a multi-stage convex relaxation approach by solving the exact penalty problem
  in an alternating way. In particular, under a restricted eigenvalue condition weaker than
  the RIP condition used in \cite{Candes11,Mohan10new}, we quantify the reduction of the error and approximate
  bounds derived from the first stage nuclear norm convex relaxation in the subsequent stages, and establish
  the geometric convergence of the error sequence in a statistical sense. Among others, the latter
  entails an upper estimation for the stage number of the convex relaxations to make the estimation error
  to reach the statistical error level. The analysis shows that the error and approximate
  rank bounds of the nuclear norm relaxation are reduced most in the second stage and the reduction rate
  is at least $40\%$ for those problems with a relatively worse restricted eigenvalue property, and the reduction
  becomes less as the number of stages increases and can be ignored after the fifth stage.
 %---------------------------------------------------------------------------------------------------Subsection 1.2
  \subsection{Related works}\label{subsec1.2}

  The idea of using the multi-stage convex relaxation for low-rank optimization problems is not new.
  In order to improve the solution quality of the nuclear norm relaxation method, some researchers pay their attention
  to nonconvex surrogates of low-rank optimization problems. Since seeking a global optimal solution of
  a nonconvex surrogate problem is almost as difficult as solving a low-rank optimization problem itself,
  they relax nonconvex surrogates into a sequence of simple matrix optimization problems,
  and develop the reweighted minimization methods (see \cite{Fazel03,Mohan12,LXW13}).
  In contrast to our multi-stage convex relaxation approach, such sequential convex relaxation methods
  are designed by solving a sequence of convex relaxation problems of nonconvex surrogates instead of
  the equivalent reformulation. We also notice that the theoretical analysis in \cite{Mohan10new}
  for the reweighted trace norm minimization method \cite{Fazel03} depends on the special property
  of the log-determinant function, which is not applicable to general low-rank optimization problems,
  and the theoretical guarantees in \cite{LXW13} were established only for the noiseless recovery problem.

  \medskip

  {Additionally}, some researchers have reformulated low-rank optimization problems as
  smooth nonconvex problems with the help of low-rank decomposition of matrices in the attempt
  to achieve a desirable solution by solving the smooth nonconvex problems in an alternating way
  (actually by solving a sequence of simple convex matrix optimization problems);
  see, e.g., \cite{RSebro-05,Jain-AM}. This class of convex relaxation methods has a theoretical guarantee,
  but is not applicable to those problems with hard constraints such as \eqref{rank-min}.

  \medskip

  Finally, it is worthwhile to point out that our multi-stage convex relaxation approach
  is highly relevant to the one proposed by Zhang \cite{Zhang10} for sparse regularization problems
  and the rank-corrected procedure  for the matrix completion problem with fixed coefficients \cite{MiaoPS16}.
  The former is designed via solving a sequence of convex relaxation problems for the nonconvex surrogates
  of the zero-norm regularization problem. Since the singular values vectors are involved in low-rank matrix recovery,
  the analysis technique in \cite{Zhang10} is not applicable to our multi-stage convex
  approach to problem \eqref{rank-min}. In particular, for low-rank matrix recovery,
  it is not clear whether the error sequence yielded by the multi-stage convex relaxation
  approach shrinks geometrically or not in a statistical sense, and if it does,
  under what conditions. We will answer these questions affirmatively in Section \ref{sec4}.
  The rank-corrected procedure \cite{MiaoPS16} is actually a two-stage convex relaxation approach
  in which the first-stage is to find a reasonably good initial estimator and the second-stage is
  to solve the rank-corrected problem. This procedure has already been applied to nonlinear
  dimensionality reduction problems \cite{Ding16} and tensor completion problems \cite{Bai16}.
  However, when the rank of the true matrix is unknown, the rank-corrected problem in \cite{MiaoPS16}
  needs to be constructed heuristically with the knowledge of the initial estimator,
  while each subproblem in our multi-stage convex relaxation approach stems from
  the global exact penalty of the equivalent MPGCC. In addition, the analytical technique used in \cite{MiaoPS16}
  is more reliant on concentration inequalities in probability analysis, whereas our analysis
  is deterministic and relies on the restricted eigenvalue property of linear operators.

%----------------------------------------------------------------------------------------------Subsection 1.2
  \subsection{Notation}\label{sec1.3}

  We stipulate $n_1\le n_2$ and let $\mathbb{R}^{n_1\times n_2}$ be the vector space of
  all $n_1\!\times\!n_2$ real matrices endowed with the trace inner product
  $\langle \cdot,\cdot\rangle$ and its induced norm $\|\cdot\|_F$.
  For $X\in\mathbb{R}^{n_1\times n_2}$, we denote $\sigma(X)\in\mathbb{R}^{n_1}$ by
  the singular value vector of $X$ with entries arranged in a non-increasing order,
  and $\|X\|_*$ and $\|X\|$ by the nuclear norm and the spectral norm of $X$, respectively.
  Let $\mathbb{O}^{n\times\kappa}$ be the set in $\mathbb{R} ^{n\times\kappa}$ consisting of
  all matrices whose columns are of unit length and are mutually orthogonal to each other,
  and denote $\mathbb{O}^{n} =\mathbb{O}^{n\times n}$. Let $e$ be the vector of all ones
  whose dimension is known from the context.

  \medskip

  Let $\Phi$ be the family of closed proper convex functions $\phi\!:\mathbb{R}\to(-\infty,+\infty]$ satisfying
  \begin{equation}\label{phi-assump}
   {\rm int}({\rm dom}\,\phi)\supseteq[0,1],\ \ 1>t^*:=\mathop{\arg\min}_{0\le t\le 1}\phi(t),\ \ \phi(t^*)=0\ \ {\rm and}\ \ \phi'_-(1)<+\infty.
  \end{equation}
   For each $\phi\in\Phi$, let $\psi:\mathbb{R}\to\mathbb{R}\cup\{+\infty\}$ be the associated
   proper closed convex function:
   \begin{equation}\label{phi-psi}
            \psi(t):=\!\left\{\!
                 \begin{array}{cl}
                  \phi(t) &\textrm{if}\ t\in [0,1],\\
                   +\infty & \textrm{otherwise}.
                 \end{array}
                 \right.
   \end{equation}
  Then from convex analysis \cite{Roc70} we know that the conjugate $\psi^*$ of $\psi$ has the properties:
  \begin{subnumcases}{}\label{dpsi-cg1}
          \partial\psi^*(t)=\big[(\psi^*)_{-}'(t),(\psi^*)_{+}'(t)\big]\subset[0,1]\quad \forall t\in\mathbb{R},\qquad\qquad\\
             (\psi^*)_{+}'(t_1)\le (\psi^*)_{-}'(t)\le (\psi^*)_{+}'(t)\le(\psi^*)_{-}'(t_2)\quad\ \forall t_1<t<t_2.
           \label{dpsi-cg2}
   \end{subnumcases}
  We also need the eigenvalues of $\mathcal{A}^*\mathcal{A}$ restricted to
  a set of low-rank matrices, where $\mathcal{A}^*$ denotes the adjoint of $\mathcal{A}$.
  To this end, for a given positive integer $k$, we define
  \begin{align}\label{rhok}
   \vartheta_{+}(k):=\!\sup_{0<{\rm rank}(X)\leq k}\!\frac{\langle X,\mathcal {A}^*\mathcal{A}(X)\rangle}{\|X\|_F^2}\ \
   {\rm and}\ \
   \vartheta_{-}(k):=\!\inf_{0<{\rm rank}(X)\leq k}\!\frac{\langle X,\mathcal {A}^*\mathcal{A}(X)\rangle}{\|X\|_F^2},
  \end{align}
  which can be viewed as the largest and the smallest rank $k$-restricted eigenvalues of $\mathcal {A}^*\mathcal{A}$, respectively.

%--------------------------------------------------------------------------------------------------section 2
  \section{Exact penalty for an equivalent reformulation}\label{sec2}

  First of all, we shall provide an equivalent reformulation of the rank minimization problem \eqref{rank-min}
  with the help of the following variational characterization of the rank function.
%-------------------------------------------------------------------------------------------------
  \begin{lemma}\label{character-rank}
   Let $\phi\in \Phi$. Then, for any given $X\in\mathbb{R}^{n_1\times n_2}$, it holds that
   \begin{equation}\label{rank-min1}
    \phi(1)\, {\rm rank}(X)=\min_{W\in\mathbb{R}^{n_1\times n_2}}\!\Big\{{\textstyle\sum_{i=1}^{n_1}}\phi(\sigma_i(W))\!:\
     \|X\|_*\!-\!\langle W,X\rangle =0,\,\|W\|\le 1\Big\}.
   \end{equation}
  \end{lemma}
  \begin{proof}
   Fix an arbitrary matrix $X\in\mathbb{R}^{n_1\times n_2}$. Write $\kappa={\rm rank}(X)$
   and assume that $X$ has the SVD as $U[{\rm Diag}(\sigma(X))\ \ 0]V^{\mathbb{T}}$, where
   $U\!=[U_1\ \ U_2]\in\mathbb{O}^{n_1}$ and $V=[V_1\ \ V_2]\in\mathbb{O}^{n_2}$ with
   $U_1\in\mathbb{O}^{n_1\times\kappa}$ and $V_1\in\mathbb{O}^{n_2\times\kappa}$.
   Let $W\in\mathbb{R}^{n_1\times n_2}$ be an arbitrary feasible point of
   the minimization problem (\ref{rank-min1}). It then follows from \cite[Equation (3.3.25)]{HJ91} that
   \begin{equation*}
     \|X\|_*=\langle W,X\rangle\leq \langle \sigma(W), \sigma(X)\rangle\leq \|\sigma(X)\|_1=\|X\|_*,
   \end{equation*}
   which implies that $\sum_{i=1}^{n_1}(1-\!\sigma_i(W))\sigma_i(X)=0$.
   Along with $\sigma_i(W)\!\in[0,1]$ for $i=1,\ldots,n_1$, we obtain that
   $\sigma_i(W)=1$ if $\sigma_i(X)\ne 0$, and $\phi(\sigma_i(W))\geq 0$ if $\sigma_i(X)=0$.
   Consequently,
   \(
     \sum_{i=1}^{n_1}\phi(\sigma_i(W))\ge \phi(1)\,{\rm rank}(X),
   \)
   i.e., $\phi(1)\, {\rm rank}(X)$ is a lower bound for the optimal value of (\ref{rank-min1}).
   Clearly, $W^*\!=U_1V_1^{\mathbb{T}}+t^*U_2[{\rm Diag}(e)\ \ 0]V_2^{\mathbb{T}}$ with $t^*$ defined
   in \eqref{phi-assump} is feasible to (\ref{rank-min1}) with the objective value being
   $\phi(1)\,{\rm rank}(X)$. This shows that $W^*$ is an optimal solution of the minimization problem
   (\ref{rank-min1}) with the optimal value equal to $\phi(1)\, {\rm rank}(X)$.
  \end{proof}

  Recall that $\phi(1)>0$ for $\phi\in\Phi$. By Lemma \ref{character-rank},
  we readily have the following result.
%------------------------------------------------------------------------------------------
 \begin{proposition}\label{equiv-prop}
  Let $\phi\in \Phi$. Then, the rank minimization problem \eqref{rank-min} is equivalent to
  \begin{align}\label{rank-MPGCC}
   &\min_{X,W\in\mathbb{R}^{n_1\times n_2}}{\textstyle\sum_{i=1}^{n_1}}\phi(\sigma_i(W))\nonumber\\
   &\qquad{\rm s.t.}\ \ \|\mathcal{A}(X)\!-\!b\|\le \delta,\,X\in\Omega,\\
   &\qquad\qquad \|X\|_*\!-\!\langle W,X\rangle=0,\,\|W\|\leq 1\nonumber
  \end{align}
  in the following sense: if $X^*=U^*[{\rm Diag}(\sigma(X^*))\ \ 0](V^*)^{\mathbb{T}}$
  is a global optimal solution of (\ref{rank-min}), where $U^*=[U_1^*\ \ U_2^*]\in\mathbb{O}^{n_1}$
  and $V^*=[V_1^*\ \ V_2^*]\in\mathbb{O}^{n_2}$ with $U_1^*\in\mathbb{O}^{n_1\times r}$
  and $V_1^*\in\mathbb{O}^{n_2\times r}$ for $r={\rm rank}(X^*)$,
  then $(X^*,U_1^*(V_1^*)^{\mathbb{T}}\!+\!t^*U_2^*[{\rm Diag}(e)\ \ 0](V_2^*)^{\mathbb{T}})$ is
  globally optimal to (\ref{rank-MPGCC}); and conversely, if $(X^*,W^*)$ is a global optimal
  solution to (\ref{rank-MPGCC}), then $X^*$ is globally optimal to (\ref{rank-min}).
 \end{proposition}

  The constraints $\|X\|_*\!-\!\langle W,X\rangle =0$ and $\|W\|\le 1$ involve a complementarity
  relation that, for the positive semidefinite (PSD) rank minimization problem, is exactly the PSD
  cone complementarity relation. In view of this, we call problem \eqref{rank-MPGCC} an MPGCC.

  \medskip

  Due to the presence of the nonconvex constraint
  \(
    \|X\|_*\!-\langle W,X\rangle\!=0,
  \)
  the MPGCC (\ref{rank-MPGCC}) is as difficult as the original problem \eqref{rank-min}.
  Nevertheless, it provides us a new view to tackle the difficult rank minimization problem \eqref{rank-min}.
  Since numerically it is usually more convenient to handle nonconvex objective functions than
  to handle nonconvex constraints, we are motivated to investigate the following penalized problem
  of the MPGCC (\ref{rank-MPGCC}):
  \begin{align}\label{rank-penalty}
   &\min_{X,W\in\mathbb{R}^{n_1\times n_2}}{\textstyle\sum_{i=1}^{n_1}}\phi(\sigma_i(W))+\rho(\|X\|_*\!-\!\langle W,X\rangle)\nonumber\\
   &\qquad{\rm s.t.}\quad\ \|\mathcal{A}(X)\!-\!b\|\le \delta,\, X\in\Omega,\,\|W\|\leq 1.
  \end{align}
  Next we shall verify that \eqref{rank-penalty} is an exact penalty version for (\ref{rank-MPGCC})
  in the sense that there exists a constant $\overline{\rho}>0$ such that the global
  optimal solution set of (\ref{rank-penalty}) associated to any $\rho>\overline{\rho}$
  coincides with that of (\ref{rank-MPGCC}). To the best of our knowledge,
  there are only a few works devoted to  mathematical programs with
  matrix cone complementarity constraints \cite{Ding13,WZZhang14}, which mainly focus on
  the optimality conditions, but not the exact penalty conditions.
%-------------------------------------------------------------------------------------- Theorem 3.1
 \begin{theorem}\label{epenalty-minrank}
  Let the optimal value of \eqref{rank-min} be $r>0$. Then, there exists a constant $\alpha>0$
  such that $\sigma_r(X)\ge\alpha$ for all $X\in\mathcal{F}$, where $\mathcal{F}$ is
  the feasible set of \eqref{rank-min}, and for $\phi\in\Phi$ the global optimal solution
  set of \eqref{rank-penalty} associated to any $\rho>\frac{\phi_{-}'(1)}{\alpha}$ is
  the same as that of (\ref{rank-MPGCC}).
 \end{theorem}
 \begin{proof}
  Suppose that there exists a sequence $\{X^k\}\subset\mathcal{F}$ such that
  $\sigma_{r}(X^k)\rightarrow 0$. Notice that $\{X^k\}$ is bounded since $\mathcal{F}$ is bounded.
  Let $\widehat{X}$ be an accumulation point of $\{X^k\}$.
  By the closedness of $\mathcal{F}$ and the continuity of $\sigma_{r}(\cdot)$,
  $\widehat{X}\in \mathcal{F}$ and $\sigma_{r}(\widehat{X})=0$, so that ${\rm rank}(\widehat{X})\le r-1$.
  This is a contradiction which establishes the existence of $\alpha$.

  \medskip

  We denote by $\mathcal{S}$ and $\mathcal{S}^*$   the feasible set
  and the global optimal solution set of (\ref{rank-MPGCC}), respectively.  For any given $\rho>0$,
  let $\mathcal{S}_{\rho}$ and $\mathcal{S}_{\rho}^*$ be the feasible set and
  the global optimal solution set of the corresponding penalty problem (\ref{rank-penalty}), respectively.
  By the first part of our conclusion, there exists a constant $\alpha>0$ such that
  $\sigma_{r}(X)\ge \alpha$ for all $X\in\mathcal{F}$. Let $\rho$ be an arbitrary constant
  with $\rho>{\phi_{-}'(1)}/{\alpha}$. Then, for any $X\in\mathcal{F}$ and each $i\in\{1,\ldots,r\}$,
  \begin{equation}\label{phi1-r}
  \{1\}={\textstyle\mathop{\arg\min}_{t\in[0,1]}}\big\{\phi(t) +\rho\sigma_i(X)(1-t)\big\}.
  \end{equation}
  First we verify that each $(X^*,W^*)\in\mathcal{S}_{\rho}^*$ satisfies
  \(
    \|X^*\|_*-\langle W^*,X^*\rangle=0
  \)
  and ${\rm rank}(X^*)=r$. Indeed, since $\mathcal{S}^*\subset \mathcal{S}\subset\mathcal{S}_{\rho}$
  and $r\phi(1)$ is the optimal value of  problem \eqref{rank-MPGCC}, it holds that
  \begin{equation}\label{temp-ineq1}
  r\phi(1)\geq {\textstyle\sum_{i=1}^{n_1}}\phi(\sigma_i(W^*)) +\rho(\|X^*\|_*-\langle W^*,X^*\rangle).
  \end{equation}
  In addition, from \cite[Equation (3.3.25)]{HJ91}, it follows that
  \begin{align*}
   {\textstyle\sum_{i=1}^{n_1}}\phi(\sigma_i(W^*)) +\!\rho(\|X^*\|_*\!-\!\langle W^*,X^*\rangle)
   &\geq{\textstyle\sum_{i=1}^{n_1}}\big[\phi(\sigma_i(W^*)) +\rho\sigma_i(X^*)(1-\sigma_i(W^*))\big]\nonumber\\
   &\geq {\textstyle\sum_{i=1}^{r}}\big[\phi(\sigma_i(W^*)) +\rho\sigma_i(X^*)(1-\sigma_i(W^*))\big]\nonumber\\
   &\geq {\textstyle\sum_{i=1}^{r}}\min_{t\in[0,1]}\big[\phi(t)\! +\!\rho\sigma_i(X^*)(1\!-\!t)\big]= r\phi(1),
  \end{align*}
  where the second inequality is by the nonnegativity of $\phi(\sigma_i(W^*))$ and
  $\sigma_i(X^*)(1-\sigma_i(W^*))$ for all $i$, and the last one is due to \eqref{phi1-r}.
  Together with \eqref{temp-ineq1}, we obtain that
  \begin{align*}
  {\textstyle\sum_{i=1}^{n_1}}\phi(\sigma_i(W^*))+\!\rho(\|X^*\|_*\!-\!\langle W^*,X^*\rangle)
   &\!=\!{\textstyle \sum_{i=1}^{r}}\big[\phi(\sigma_i(W^*)) +\rho\sigma_i(X^*)(1-\sigma_i(W^*))\big]\nonumber\\
   &=\!{\textstyle\sum_{i=1}^{r}}\min_{t\in[0,1]}\big[\phi(t)\! +\!\rho\sigma_i(X^*)(1\!-\!t)\big]= r\phi(1).
  \end{align*}
  This, along with \eqref{phi1-r}, implies that $\sigma_i(W^*)=1$ for $i=1,\ldots,r$.
 Substituting $\sigma_i(W^*)=1$ for $i=1,\ldots,r$ into the last equation and using
 the nonnegativity of $\phi$ in $[0,1]$, we deduce that $\sum_{i=r+1}^{n_1}\phi(\sigma_i(W^*))=0$
 and $\|X^*\|_*=\langle W^*,X^*\rangle=\langle \sigma(X^*), \sigma(W^*)\rangle$.
 This means that $\sigma_i(W^*)=t^*$ for $i=r+\!1,\ldots,n_1$
 and ${\rm rank}(X^*)= r$, where  $t^*<1$ is defined in (\ref{phi-assump}).
 Then, $\mathcal{S}_{\rho}^*\subset \mathcal{S}$ and $\sum_{i=1}^{n_1}\phi(\sigma_i(W^*))=r\phi(1)$
 for $(X^*,W^*)\in \mathcal{S}_{\rho}^*$. Since the global optimal value of \eqref{rank-MPGCC} is $r\phi(1)$,
 we have $\mathcal{S}_{\rho}^*\subseteq \mathcal{S}^*$. For the reverse inclusion,
 let $(X^*,W^*)$ be an arbitrary point from $\mathcal{S}^*$.
 Then $(X^*,W^*)\in \mathcal{S}_{\rho}$ and $\sum_{i=1}^{n_1}\phi(\sigma_i(W^*))=r\phi(1)$.
 While the optimal value of \eqref{rank-penalty} is $r\phi(1)$ by the last equation.
 Thus, we get $S^*\subseteq\mathcal{S}^*_{\rho}$.
 \end{proof}

  Theorem \ref{epenalty-minrank} extends the exact penalty result of \cite[Theorem 3.3]{BiPan14}
  for the zero-norm minimization problem to the matrix setting, and further develops
  the exact penalty result of the rank-constrained minimization problems in \cite[Theorem 3.1]{BiPan16}.
  Observe that the objective function of problem (\ref{rank-penalty}) is globally Lipschitz continuous
  over its feasible set. Combining Theorem \ref{epenalty-minrank} with Proposition \ref{equiv-prop},
  we conclude that the rank minimization problem \eqref{rank-min} is equivalent to the Lipschitz
  optimization problem \eqref{rank-penalty}.

%-----------------------------------------------------------------------------------------------------Section 3
  \section{A multi-stage convex relaxation approach}\label{sec3}

  In the last section, we prove that the rank minimization problem \eqref{rank-min}
  is equivalent to a single penalty problem (\ref{rank-penalty}). This penalty problem
  depends on the parameter $\alpha$, the lower bound for the $r$th largest singular value of
  all $X\in\mathcal{F}$, which can be difficult to estimate. This means that a sequence
  of penalty problems of the form (\ref{rank-penalty}) with non-decreasing $\rho$ should be solved
  so as to target achieving a global optimal solution of \eqref{rank-min}. The penalty problem
  \eqref{rank-penalty} associated to a given $\rho>0$ is not globally solvable due to
  the nonconvexity of the objective function. However, it becomes a nuclear semi-norm
  minimization with respect to $X$ if the variable $W$ is fixed and has a closed form solution
  of $W$ (as will be shown later) if the variable $X$ is fixed. This motivates us to
  propose a multi-stage convex relaxation approach to \eqref{rank-min} by solving
  a single penalty problem (\ref{rank-penalty}) in an alternating way.

  \bigskip
  \setlength{\fboxrule}{0.5pt}
  \noindent
  \fbox{
  \parbox{0.96\textwidth}
  {
%----------------------------------------------------------------------------------------------Algorithm
 \begin{algorithm} \label{Alg1}({\bf A multi-stage convex relaxation approach})
 \begin{description}
 \item[(S.0)] Choose a function $\phi\in\Phi$. Let $W^0:=0$ and set $k:=1$.

 \item[(S.1)] Solve the following nuclear semi-norm minimization problem
              \begin{equation}\label{subprob-X}
                 X^{k}\in \mathop{\arg\min}_{X\in\mathbb{R}^{n_1\times n_2}}
                 \big\{\|X\|_*\!-\!\langle W^{k-1},X\rangle\!:\ \|\mathcal{A}(X)\!-b\|\le \delta,\ X\in\Omega\big\}.
              \end{equation}
              \hspace*{0.1cm} If $k=1$, select a suitable $\rho_1>0$ and go to Step (S.3); else go to Step (S.2).

  \item[(S.2)] Select a suitable ratio factor $\mu_k\ge 1$ and
               set $\rho_{k}:=\mu_k\rho_{k-1}$.

  \item[(S.3)] Solve the following minimization problem
               \begin{equation}\label{subprob-W}
                W^{k}\in \mathop{\arg\min}_{W\in\mathbb{R}^{n_1\times n_2}}
                \big\{{\textstyle\sum_{i=1}^{n_1}}\phi(\sigma_i(W))\!-\rho_k\langle W,X^k\rangle\!:\ \|W\|\le 1\big\}.
               \end{equation}
  \item[(S.4)] Let $k\leftarrow k+1$, and then go to Step (S.1).
 \end{description}
 \end{algorithm}}
 }

  \bigskip

  The subproblem (\ref{subprob-X}) corresponds to the penalty problem (\ref{rank-penalty}) associated
  to $\rho_{k-1}$ with the variable $W$ fixed to $W^{k-1}$. Since the set $\Omega$ is assumed to be compact,
  its solution $X^{k}$ is well defined. Assume that $X^{k}$ has the SVD as
  $U^{k}[{\rm Diag}(\sigma(X^{k}))\ \ 0](V^{k})^{\mathbb{T}}$. By \cite[Eq.(3.3.25)]{HJ91},
  it is easy to check that if $z^*$ is optimal to the convex minimization
  \begin{equation}\label{zmin-prob}
    \min_{z\in\mathbb{R}^{n_1}}\left\{{\textstyle\sum_{i=1}^{n_1}}\psi(z_i)-\rho\langle z,\sigma(X^k)\rangle\right\},
  \end{equation}
  then $Z^*=U^{k}[{\rm Diag}(z^{*})\ \ 0](V^{k})^{\mathbb{T}}$ is a global optimal solution to \eqref{subprob-W};
  and conversely if $W^*$ is globally optimal to \eqref{subprob-W}, then $\sigma(W^*)$ is an optimal solution
  to \eqref{zmin-prob}. Write
  \begin{equation}\label{Wk}
    W^{k}:= U^{k}[{\rm Diag}\big(w_1^{k},w_2^{k},\ldots,w_{n_1}^{k})\ \ 0] (V^{k})^{\mathbb{T}}\ \ {\rm with}\ \
     w_i^{k}\in\partial\psi^*(\rho_{k}\sigma_i(X^{k})).
  \end{equation}
  Then, together with \cite[Theorem 23.5]{Roc70}, it follows that such $W^k$ is optimal to
  the subproblem (\ref{subprob-W}). This means that the main computational work of Algorithm \ref{Alg1}
  consists of solving a sequence of subproblems (\ref{subprob-X}). Unless otherwise stated,
  in the sequel we choose $w_i^k=w_j^k$ when $\sigma_i(X^{k})=\sigma_j(X^{k})$,
  which ensures that $1\ge w_1^k\ge\cdots\ge w_{n_1}^k\ge 0$.

  \medskip

   Since $\| W^{k-1}\|\leq 1$, the function $\|\cdot\|_*-\langle W^{k-1},\cdot\rangle$ defines
  a semi-norm over $\mathbb{R}^{n_1\times n_2}$. So, the subproblem (\ref{subprob-X}) is
  a nuclear semi-norm minimization problem. When $k=1$, it reduces to the nuclear norm minimization problem,
  i.e., the first stage of Algorithm \ref{Alg1} is exactly the nuclear norm convex relaxation.
  It should be emphasized that Algorithm \ref{Alg1} is different from the reweighted trace norm
  minimization method \cite{Fazel03,Mohan10new} and the iterative reweighted algorithm
  \cite{LXW13}. The former is proposed from the primal and dual viewpoint by solving
  an equivalent Lipschitz reformulation in an alternating way,
  whereas the latter is proposed from the primal viewpoint by relaxing a smooth nonconvex surrogate
  of \eqref{rank-min}.

  \medskip

  To close this section, we illustrate the choice of $w_i^k$ in formula \eqref{Wk}
  with two $\phi\in\Phi$.
%---------------------------------------------------------------------------------------Example1
 \begin{example}\label{example1}
  Let $\phi_1(t)= t$ for $t\in\mathbb{R}$. Clearly, $\phi_1\in \Phi$ with $t^*=0$.
  Moreover, for the function $\psi_1$ defined by \eqref{phi-psi} with $\phi_1$,
  an elementary calculation yields that
  \begin{equation}\label{hard-subdiff}
    \psi_1^*(s)
    =\left\{\!\begin{array}{cl}
                      s\!-\!1 & \textrm{if}\ s>1;\\
                      0 & \textrm{if}\ s\leq 1
                \end{array}\right.\ {\rm and}\ \
     \partial\psi_1^*(s)=\left\{\begin{array}{cl}
                      \{1\} & \textrm{if}\ s>1;\\
                      \,[0,1] & \textrm{if}\ s=1;\\
                      \{0\} & \textrm{if}\ s<1.\\
                \end{array}\right.
  \end{equation}
  Thus, one may choose
  \(
    w_i^{k}\!=\!\left\{\begin{array}{ll}
                      1 & \textrm{if}\ \sigma_i(X^{k})\ge \frac{1}{\rho_{k}};\\
                      0 & \textrm{otherwise}
                \end{array}\right.
  \) for the matrix $W^k$ in formula (\ref{Wk}).
 \end{example}
%--------------------------------------------------------------------------------------------Example3
 \begin{example}\label{example2}
  Let $\phi_2(t)=-t-\frac{q-1}{q}(1-t+\epsilon)^{\frac{q}{q-1}}+\epsilon+\frac{q-1}{q}$
  for $t\in(-\infty,1+\epsilon)$ with $0<q<1$, where $\epsilon\in(0,1)$ is a constant.
  One can check that $\phi_2\in \Phi$ with $t^*=\epsilon$. For the function $\psi_2$
  defined by the equation \eqref{phi-psi} with $\phi_2$, an elementary calculation yields that
 \[
  \partial\psi_2^*(s)=\left\{\!\begin{array}{cl}
                      \{1\} & \textrm{if}\ s\geq \epsilon^{\frac{1}{q-1}}-1;\\
                      \{1+\epsilon-(s\!+\!1)^{q-1}\} & \textrm{if}\ (1\!+\!\epsilon)^{\frac{1}{q-1}}\!-\!1<s<\!\epsilon^{\frac{1}{q-1}}\!-\!1;\\
                      \{0\} & \textrm{if}\ s\leq (1\!+\!\epsilon)^{\frac{1}{q-1}}-1.
                \end{array}\right.
  \]
  Hence, one may take
  \(
   w_i^{k}=\min\big[1+\epsilon-(\rho_{k}\sigma_i(X^{k})\!+\!1)^{q-1},1\big]
  \)
  for the matrix $W^k$ in (\ref{Wk}),
 \end{example}

  \begin{remark}
  A constant $\epsilon\in(0,1)$ is introduced in the function $\phi_2$ so as to ensure that
  $(\phi_2)_{-}'(1)<+\infty$, and then problem (\ref{rank-penalty}) is a global exact penalization
  of (\ref{rank-min}). Thus, once $(\widehat{X},\widehat{W})$ yielded by Algorithm \ref{Alg1}
  satisfies $\|X\|_*\!-\!\langle X,W\rangle=0$, $\widehat{X}$ is at least a local minimum of
  the problem (\ref{rank-min}) since each feasible solution of \eqref{rank-min} is locally optimal.
  \end{remark}

%-----------------------------------------------------------------------------------------------------Section 4
 \section{Theoretical guarantees of Algorithm \ref{Alg1}}\label{sec4}

 In this section, we shall provide the theoretical guarantees of Algorithm \ref{Alg1} under
 a mild condition for the restricted eigenvalues of $\mathcal{A}^*\mathcal{A}$, which is stated as follows.
%---------------------------------------------------------------------------------------------
  \begin{assumption}\label{assump}
   There exist  a constant $c\in[0,\sqrt{2})$ and an integer $s\in[1,\frac{n-2r}{2}]$ such that
   $\frac{\vartheta_{+}(s)}{\vartheta_{-}(2r+2s)}\le 1+\frac{2c^2s}{r}$, where $\vartheta_{+}(\cdot)$
   and $\vartheta_{-}(\cdot)$ are the functions defined by \eqref{rhok}.
  \end{assumption}

  Assumption \ref{assump} requires the restricted eigenvalue ratio of $\mathcal{A}^*\mathcal{A}$
  to grow sublinearly in $s$. This condition, extending the sparse eigenvalue condition
  used for the analysis of sparse regularization (see \cite{Zhang09,Zhang10}),
  is weaker than the RIP condition $\delta_{4r}<\sqrt{2}-1$ used in \cite{Candes11} for $n\ge 4r$,
  where $\delta_{kr}$ is the $kr$-restricted isometry constant of $\mathcal{A}$ defined as in \cite{Candes11}.
  Indeed, from the definitions of $\vartheta_{+}(\cdot)$ and $\vartheta_{-}(\cdot)$,
  it is immediate to have that
  \[
    \frac{\vartheta_{+}(r)}{\vartheta_{-}(2r+2r)}
    \leq \frac{1+\delta_{4r}}{1-\delta_{4r}}
    <1+\frac{2\sqrt{2}\!-\!2}{2\!-\!\sqrt{2}}<1+2\times 0.843^2.
  \]
  This shows that $c=0.843$ is such that
  \(
   \frac{\vartheta_{+}(s)}{\vartheta_{-}(2r+2s)}\le 1+\frac{2c^2s}{r}
  \)
  for $s=r$. In addition, this condition is also weaker than the RIP condition
  $\delta_{3r}<2\sqrt{5}\!-\!4$ used in \cite{Mohan10new} for $n\ge 3r$, where
  $r$ is an even number or $r$ is an odd number greater than $11$. To see this,
  assume that $\delta_{3r}<2\sqrt{5}\!-\!4$, and $r$ is an even number or
  is an odd number greater than $11$. Then,
  \begin{equation}\label{relation}
    \max\left(\frac{\vartheta_{+}(r/2)}{\vartheta_{-}(2r\!+\!r)},\frac{\vartheta_{+}((r\!-\!1)/2)}{\vartheta_{-}(2r\!+\!r\!-\!1)}\right)
    \leq \frac{1+\delta_{r/2}}{1-\delta_{3r}}
    \le \frac{1+\delta_{3r}}{1-\delta_{3r}}<1+\frac{4\sqrt{5}\!-\!8}{5\!-\!2\sqrt{5}}.
  \end{equation}
  So, $c=1.34$ and $1.403$ are respectively such that
  \(
   \frac{\vartheta_{+}(s)}{\vartheta_{-}(2r+2s)}\le 1+\frac{2c^2s}{r}
  \)
  for $s=\frac{r}{2}$ and $\frac{r-1}{2}$.

  \medskip

  In the sequel, we let $\overline{X}$ have the SVD as $\overline{U}[{\rm Diag}(\sigma(\overline{X}))\ \ 0]\overline{V}^\mathbb{T}$,
  where $\overline{U}=[\overline{U}_1\ \ \overline{U}_2]\in\mathbb{O}^{n_1}$
  and $\overline{V}=[\overline{V}_1\ \ \overline{V}_2]\in\mathbb{O}^{n_2}$ with $\overline{U}_1\in\mathbb{O}^{n_1\times r}$
  and $\overline{V}_1\in\mathbb{O}^{n_2\times r}$ for $r={\rm rank}(\overline{X})$,
  and write $\mathcal {T}=\mathcal{T}(\overline{X})$ where $\mathcal{T}(\overline{X})$ is the tangent space
  at $\overline{X}$ associated to the rank constraint ${\rm rank}(X)\le r$ (see equation \eqref{Tangent-space}
  for its definition). For convenience, for $k=1,2,\ldots$, let
  \begin{equation}\label{gammak+g}
    \gamma_{k-1}:=\frac{\|\mathcal{P}_{\mathcal {T}}(W^{k-1}-\overline{U}_1\overline{V}_1^\mathbb{T})\|_F}{\sqrt{2r}
    (1-\|\mathcal{P}_{\mathcal{T}^{\perp}}(W^{k-1})\|)}.
  \end{equation}
  The proofs of all the results in the subsequent subsections are given in Appendix C.

%--------------------------------------------------------------------------------------------Subsection
  \subsection{Error and approximate rank bounds}\label{sec4.1}

  Under Assumption \ref{assump}, when $\gamma_{k-1}\in[0,1/c)$ for some $k\ge 1$,
  we can establish the following error bound and approximate rank bound for
  the solution $X^k$ of the $k$th subproblem.
%--------------------------------------------------------------------------------------------------------Corollary
  \begin{proposition}\label{prop1-sec41}
   Suppose that Assumption \ref{assump} holds and $0\le\gamma_{k-1}<{1}/{c}$ for some $k\!\ge 1$. Then
   \begin{equation}\label{noisyless-bound}
     \big\|X^{k}\!-\!\overline{X}\big\|_F
     \le \Xi(\gamma_{k-1})\ \ {\rm and}\ \
     \big\|\mathcal{P}_{\mathcal {T}^{\perp}}(X^{k})\big\|_*
     \le\Gamma(\gamma_{k-1}),
   \end{equation}
   where $\Xi\!:[0,1/c)\to\mathbb{R}_{+}$ and $\Gamma\!:[0,1/c)\to\mathbb{R}_{+}$ are
   the increasing functions defined by
   \[
      \Xi(t):=\frac{2\delta\sqrt{\vartheta_{+}(2r\!+\!s)}}{\vartheta_{-}(2r\!+\!s)}\cdot\frac{1}{1\!-\!ct}\sqrt{1\!+\!\frac{rt^2}{2s}}
     \ \ {\rm and}\ \ \Gamma(t):=\frac{2\delta\sqrt{\vartheta_{+}(2r\!+\!s)}}{\vartheta_{-}(2r\!+\!s)}\cdot\frac{\sqrt{2r}t}{1\!-\!ct}.
   \]
 \end{proposition}

%-------------------------------------------------------------------------------------------------
 \begin{remark}\label{remark-bound}
  {\bf(a)} Since $\|\mathcal{P}_{\mathcal {T}^{\perp}}(X^{k})\|_*\!=0$ implies that ${\rm rank}(X^k)\!\le 2r$,
  it is reasonable to view $\|\mathcal{P}_{\mathcal {T}^{\perp}}(X^{k})\|_*$ as a measure for
  the approximate rank of $X^k$. So, the second inequality in \eqref{noisyless-bound}
  provides an approximate rank bound for $X^k$. The error and approximate rank bounds
  in \eqref{noisyless-bound} consist of two parts: one part is the statistical error
  $\Xi(0)=\frac{2\delta\sqrt{\vartheta_{+}(2r\!+\!s)}}{\vartheta_{-}(2r\!+\!s)}$
  from the noise and the operator $\mathcal{A}$, and the other part is the estimation error from $\gamma_{k-1}$.

  \medskip
  \noindent
  {\bf(b)} Since $W^0=0$, we have $\gamma_0=\frac{1}{\sqrt{2r}}\|\overline{U}_1\overline{V}_1^\mathbb{T}\|_F
  =\!\frac{1}{\sqrt{2}}<\frac{1}{c}$. Hence, under Assumption \ref{assump},
  the error and approximate rank bounds of the nuclear norm convex relaxation are
  \begin{equation}\label{Ebound-Nnorm}
   \big\|X^1\!-\!\overline{X}\big\|_F
   \le\Xi(\gamma_0)=\Xi(1/\sqrt{2})\ \ {\rm and}\ \
   \big\|\mathcal{P}_{\mathcal {T}^{\perp}}(X^1)\big\|_*
   \le\Gamma(\gamma_0)=\Gamma(1/\sqrt{2}).
  \end{equation}
  Moreover, if Assumption \ref{assump} is satisfied with $s=r/2$ and
  $c<\sqrt{2}-\frac{2(1-\delta_{3r}(1+{\sqrt{5}}/{2}))}{\sqrt{{3}}(1-\delta_{3r})}$
  for $\delta_{3r}<2\sqrt{5}\!-\!4$, then the error bound $\Xi(\gamma_{0})$ is tighter than
  the bound
  \(
    \frac{3\delta\sqrt{1+\delta_{3r}}}{1-\delta_{3r}(1+{\sqrt{5}}/{2})}
  \)
 given by \cite[Theorem III.1]{Mohan10new} with $C_{1,1}=1$ for the nuclear norm relaxation because
 \begin{align*}\label{comparing-ineq}
  \Xi(\gamma_{0})
  =\frac{\sqrt{\vartheta_{+}(2.5r)}\sqrt{6}\delta}{\big(1\!-\!{c}/{\sqrt{2}}\big)\vartheta_{-}(2.5r)}\leq \frac{\sqrt{1+\delta_{3r}}\sqrt{6}\delta}{\big(1\!-\!{c}/{\sqrt{2}}\big)(1-\delta_{3r})}
     <\frac{3\delta\sqrt{1+\delta_{3r}}}{1\!-\!\delta_{3r}(1\!+\!{\sqrt{5}}/{2})}.
 \end{align*}
 \end{remark}

 Remark \ref{remark-bound} (b) says  that under Assumption \ref{assump} the solution $X^1$ of
  the first stage convex relaxation has the error and approximate rank bounds as in \eqref{Ebound-Nnorm}.
  However, it is not clear whether $X^k\ (k\ge 2)$ has such error and approximate rank bounds or not.
  The following theorem states that if in addition $\sigma_r(\overline{X})>2\Xi(\gamma_0)$
  and $\rho_1$ and $\mu_k$ are appropriately chosen, all $X^k\ (k\ge 2)$ have the bounds as in \eqref{noisyless-bound},
  and more importantly, their error and approximate rank bounds are, respectively, smaller than those of $X^1$.
  To achieve this result, we need the sequence $\{\widetilde{\gamma}_k\}_{k\ge 1}$,
  which is defined recursively with $\widetilde{\gamma}_0=\gamma_0$ as
  \begin{subnumcases}{}\label{wabk}
   \widetilde{\gamma}_{k}
   :=\frac{\sqrt{r}(1-\widetilde{b}_k)+ (\sqrt{2}\widetilde{a}_k +1)\widetilde{\beta}_k}
     {\sqrt{2r}(1-\widetilde{a}_k)(1\!-\!\widetilde{\beta}_k^2)}\ \ {\rm with}\ \
    \widetilde{a}_k=(\psi^*)_{+}'\left[\rho_{k}\Xi(\widetilde{\gamma}_{k-1})\right],\\
   \widetilde{b}_k=(\psi^*)_{-}'\left[\rho_{k}(\sigma_r(\overline{X})\!-\!\Xi(\widetilde{\gamma}_{k-1}))\right],\
    \widetilde{\beta}_k=-\frac{1}{\sqrt{2}}\ln\Big[1-\frac{\sqrt{2}\Xi(\widetilde{\gamma}_{k-1})}{\sigma_r(\overline{X})}\Big].
    \label{wgammak}
  \end{subnumcases}

  \vspace{-0.3cm}
%--------------------------------------------------------------------------------------------Theorem
  \begin{theorem}\label{theorem1-sec41}
  Suppose that Assumption \ref{assump} holds and $\sigma_r(\overline{X})> 2\Xi(\gamma_0)$.
   If the parameters $\rho_1$ and $\mu_k$ are respectively chosen such that
    \(
    \widetilde{a}_1<\!\frac{(\widetilde{b}_1-\widetilde{\beta}_1^2)\sqrt{r}-\widetilde{\beta}_1}{(1-\widetilde{\beta}_1^2)\sqrt{r}+\sqrt{2}\widetilde{\beta}_1}
   \)
  and $\mu_{k}\in\big[1,\frac{\Xi(\widetilde{\gamma}_{k-2})}{\Xi(\widetilde{\gamma}_{k-1})}\big]$,
  then all $X^k\ (k\ge 1)$ satisfy the inequalities in \eqref{noisyless-bound} and for $k\ge 2$ it also holds that
  \begin{align}
   \big\|X^k\!-\!\overline{X}\big\|_F
   \le \Xi(\gamma_{k-1})\le \Xi(\widetilde{\gamma}_{k-1})<\Xi(\widetilde{\gamma}_{k-2})<\cdots<\Xi(\gamma_0),\nonumber\\
   \|\mathcal{P}_{\mathcal{T}^{\perp}}(X^{k})\|_*
    \le \Gamma(\gamma_{k-1})\le   \Gamma(\widetilde{\gamma}_{k-1})<\Gamma(\widetilde{\gamma}_{k-2})<\cdots<\Gamma(\gamma_0).\nonumber
  \end{align}
  \end{theorem}
%----------------------------------------------------------------------------------------------
  \begin{remark}\label{remark42}
  {\bf (a)} Theorem \ref{theorem1-sec41} shows that under Assumption \ref{assump}
  and $\sigma_r(\overline{X})\!>\!2\Xi(\gamma_0)$, if $\rho_1$ and $\mu_k$ are
  chosen appropriately, then the error and approximate rank bounds of $X^k\ (k\ge 2)$
  improve those of $X^1$ at least by $1-\frac{\Xi(\widetilde{\gamma}_{k-1})}{\Xi(\gamma_0)}$
  and $1\!-\!\frac{\Gamma(\widetilde{\gamma}_{k-1})}{\Gamma(\gamma_0)}$, respectively.

  \medskip
  \noindent
  {\bf(b)} The choice of $\rho_1$ depends on $\Xi(\gamma_0)$. For instance, take the function $\phi_1$
  in Example \ref{example1}. If $\sigma_r(\overline{X})=\alpha\Xi(\gamma_0)$
  for $\alpha\ge 2.5$, then by virtue of the definitions of $\widetilde{a}_1,\widetilde{b}_1$ and
  $\widetilde{\beta}_1$ and equation \eqref{hard-subdiff}
  it is easy to check that $\widetilde{a}_1=0,\widetilde{b}_1=1$ and $\widetilde{\beta}_1\in[0,0.6)$,
  and consequently $(\widetilde{b}_1-\widetilde{\beta}_1^2)\sqrt{r}-\widetilde{\beta}_1>0$.
  This means that $\big(\frac{1}{(\alpha-1)\Xi(\gamma_0)},\frac{1}{\Xi(\gamma_0)}\big)$
  is the range of choice for $\rho_1$. For numerical computations, one may estimate $r$ and
  $\sigma_r(\overline{X})$ with the help of $\sigma(X^1)$.
  \end{remark}

  To close this subsection, we illustrate the ratios $\frac{\Xi(\widetilde{\gamma}_{k-1})}{\Xi(\gamma_0)}$ and
  $\frac{\Gamma(\widetilde{\gamma}_{k-1})}{\Gamma(\gamma_0)}$ by the functions $\phi_1$ and $\phi_2$
  with $q=1/2$ and $\epsilon=10^{-3}$. For this purpose, we suppose that Assumption \ref{assump} holds
  with $r=10, s\!=\!r/2$ and $\sigma_r(\overline{X})=\alpha\Xi(\gamma_0)$ for $\alpha\!\ge\! 4.5$.
  Then, for those $c$ in the first row of Table \ref{tab1}, one may compute the ratios
  $\frac{\Xi(\widetilde{\gamma}_{k-1})}{\Xi(\gamma_0)}$ and $\frac{\Gamma(\widetilde{\gamma}_{k-1})}{\Gamma(\gamma_0)}$
  as those in the last six columns of Table \ref{tab1} with $\rho_1$ chosen as
  the middle point of the interval and $\mu_k\equiv 1$. We see that the error bound of
  the first stage is reduced most in the second stage, and as the number of stages increases,
  the reduction becomes less. For Algorithm \ref{Alg1} with $\phi_1$, the reduction is close to
  the limit ${\Xi(0)}/{\Xi(\widetilde{\gamma}_0)}$ when $k=5$, but for Algorithm \ref{Alg1}
  with $\phi_2$, there is a little room for the reduction especially for those
  $\mathcal{A}^*\mathcal{A}$ with $c\ge 0.5$.

%------------------------------------------------------------------------------------Table 1
 \begin{table}[htbp]
  \begin{center}
    {\caption{\label{tab1} Reduction rate of the error bounds of the first stage in the 2nd-$5$th stage}}
 \begin{tabular}{|c|c||ccccccc|}
   \hline
  \ & $\rho_1$    & $c$ & 0 & 0.1 & 0.3 & 0.5 & 0.7 &0.9\\
  \hline
  \raisebox{2.0ex}
  \ $\phi_1$ &$\big[\frac{0.29\alpha}{\sigma_r(\overline{X})},\frac{\alpha}{\sigma_r(\overline{X})}\big)$
   & ${\Xi(\widetilde{\gamma}_1)}/{\Xi(\gamma_0)}$         & 0.819 & 0.766 & 0.658 & 0.547 & 0.433 &0.316\\
   \    & & ${\Xi(\widetilde{\gamma}_2)}/{\Xi(\gamma_0)}$  & 0.818 & 0.763 & 0.652 & 0.537 & 0.420 &0.302\\
   \     & & ${\Xi(\widetilde{\gamma}_3)}/{\Xi(\gamma_0)}$ & 0.818 & 0.763 & 0.651 & 0.536 & 0.420 &0.302\\
   \   & & ${\Xi(\widetilde{\gamma}_4)}/{\Xi(\gamma_0)}$   & 0.818 & 0.763 & 0.651 & 0.536 & 0.420 &0.302\\
   \hline
   \raisebox{2.0ex}
  \ $\phi_2$ &$\big[\frac{0.24\alpha}{\sigma_r(\overline{X})},\frac{4.42\alpha}{\sigma_r(\overline{X})}\big]$
      & ${\Xi(\widetilde{\gamma}_1)}/{\Xi(\gamma_0)}$        & 0.975  & 0.969 & 0.955 & 0.934 & 0.905 & 0.856\\
     \   &  & ${\Xi(\widetilde{\gamma}_2)}/{\Xi(\gamma_0)}$ & 0.967   & 0.958 & 0.931 & 0.888 & 0.816 & 0.689\\
      \  &   & ${\Xi(\widetilde{\gamma}_3)}/{\Xi(\gamma_0)}$ & 0.965  & 0.954 & 0.920 & 0.760 & 0.752 & 0.572\\
       \  &   & ${\Xi(\widetilde{\gamma}_4)}/{\Xi(\gamma_0)}$ & 0.965 & 0.953 & 0.915 & 0.744 & 0.714 & 0.516\\
   \hline
   \raisebox{2.0ex}
  \    &  & ${\Xi(0)}/{\Xi(\gamma_0)}$ & 0.817 & 0.759 & 0.644 & 0.528 & 0.413 & 0.297\\
   \hline
  \end{tabular}
  \end{center}
  \end{table}

 %-----------------------------------------------------------------------------------------------Subsection 4.2
 \subsection{Geometric convergence}\label{sec4.2}

  Generally speaking, because of the presence of the noise, it is impossible for the error sequence
  $\{\|X^{k}\!-\!\overline{X}\|_F\}_{k\ge 1}$ to decrease and then converge geometrically.
  However, one may establish its geometric convergence in a statistical sense
  as in the following theorem.
%---------------------------------------------------------------------------------------------Theorem
 \begin{theorem}\label{theorem1-sec42}
  Suppose that Assumption \ref{assump} holds and $\sigma_r(\overline{X})\!>\max(2,\sqrt{2}\!+\!\alpha)\Xi(\gamma_0)$
  for $\alpha=\frac{1+\sqrt{2}\widetilde{a}_1}{(1-\widetilde{a}_{1})(1\!-\!\widetilde{\beta}_{1}^2)\sqrt{r+4s}}$.
  If $\rho_1$ and $\mu_k$ are chosen as in Theorem \ref{theorem1-sec41}, then for $k\ge 1$,
  \begin{equation}\label{gconverge1}
   \big\|X^k\!-\!\overline{X}\big\|_F
   \le\!\frac{\Xi(0)}{1\!-\!c\widetilde{\gamma}_1}
        \Big[1+\frac{(1-\widetilde{b}_{1})\sqrt{r}}
            {2(1\!-\!\widetilde{a}_{1})(1\!-\!\widetilde{\beta}_{1}^2)\sqrt{s}}\Big]
            +\Big[\frac{\alpha\,\Xi(\gamma_0)}{\sigma_{r}(\overline{X})\!-\!\sqrt{2}\Xi(\gamma_0)}\Big]^{k-1}\big\|X^{1}\!-\!\overline{X}\big\|_F.
  \end{equation}
  \end{theorem}
%----------------------------------------------------------------------------------------
  \begin{remark}\label{remark-Gconvergence}
  {\bf(a)} The requirement $\sigma_r(\overline{X})\!>\max(2,\sqrt{2}\!+\!\alpha)\Xi(\gamma_0)$
  in Theorem \ref{theorem1-sec42} is bit stronger than $\sigma_r(\overline{X})\!>2\Xi(\gamma_0)$.
  Take $\phi_1$ for example. When $\sigma_r(\overline{X})\ge 2.4\Xi(\gamma_0)$,
  this requirement is automatically satisfied. Also, now we have that
  $\varrho:=\frac{\alpha\,\Xi(\gamma_0)}{\sigma_{r}(\overline{X})\!-\!\sqrt{2}\Xi(\gamma_0)}\le 0.76$.

  \medskip
  \noindent
  {\bf(b)} The first term of the sum on the right hand side of \eqref{gconverge1} represents
  the statistical error arising from the noise and the sampling operator $\mathcal{A}$,
  and the second term is the estimation error related to the multi-stage convex relaxation.
  Clearly, the statistical error is of a certain order of $\Xi(0)$. Thus, to guarantee that
  the second term is less than the statistical error, at most $\overline{k}$ stage
  convex relaxations are required, where
  \[
    \overline{k}=\frac{\log(\Xi(0))-\log(\|X^1-\overline{X}\|_F)}{\log\varrho}+1
    \le \frac{\log(\Xi(0)/\Xi(\gamma_0))}{\log\varrho}+1.
  \]
   Take $\varrho=0.7$ for example. When $s=r$, one can calculate that $\overline{k}\le 2$ if $c=0.3$,
  and $\overline{k}\le 4$ if $c=0.7$. This means that, for those $\mathcal{A}^*\mathcal{A}$ with
  a worse restricted eigenvalue condition, more than two stage convex relaxations are needed to
  yield a satisfactory solution.
  \end{remark}

  For the analysis in the previous two subsections, the condition
   $\sigma_r(\overline{X})\!\ge\alpha\Xi(\gamma_0)$ for a certain $\alpha> 2$
  is required for the decreasing of the error and approximate rank bounds
  of the first stage convex relaxation and the contraction of the error sequence.
  Such a condition is necessary for the low-rank recovery since, when the smallest
  nonzero singular is mistaken as a zero, the additional singular vectors will yield
  a large error. In fact, in the geometric convergence analysis of sparse vector optimization
  (see \cite{Zhang10}), the error bound of the first stage was implicitly assumed not to be too large.
  In addition, we observe that the structure information of $\overline{X}$ does not lend any help
  to the low-rank matrix recovery in terms of convergence rates. However, when the true matrix
  has a certain structure, it is necessary to incorporate such structure information
  into model \eqref{rank-min}. Otherwise, the solution $X^k$ yielded by
  the multi-stage convex relaxation may not satisfy the structure constraint, and then
  it is impossible to control the error of $X^k$ to the true matrix $\overline{X}$.

  \medskip

  Finally, we point out that when the components $\xi_1,\xi_2,\ldots,\xi_m$ of
  the noisy vector $\xi$ are independent (but not necessarily identically distributed) sub-Gaussians,
  i.e., there exists a constant $\sigma\ge 0$ such that $\mathbb{E}[e^{t\xi_i}]\le e^{\sigma^2t^2/2}$
  holds for all $i$ and any $t\in\mathbb{R}$, by Lemma \ref{lemma-noise} in Appendix C,
  the conclusions of Theorems \ref{theorem1-sec41} and \ref{theorem1-sec42} hold
  with $\delta=\sqrt{m}\sigma$ with probability  at least $1-\exp(1-\frac{c_1m}{4})$
  for an absolute constant $c_1>0$.  For the random $\mathcal{A}$,
  the following result is immediate by \cite[Theorem 2.3]{Candes11} and
  the first inequality in \eqref{relation}.
%-------------------------------------------------------------------------------------
 \begin{theorem}\label{random-Aoper}
  Fix $\overline{\delta}\in(0,1/2)$ and let $\mathcal{A}$ be a random measurement ensemble obeying
  the following conditions: for any given $X\in\mathbb{R}^{n_1\times n_2}$ and any fixed $0<t<1$,
  \begin{equation}\label{Acond}
   \mathbb{P}\left\{|\|\mathcal{A}(X)\|^2-\|X\|_F^2|>t\|X\|_F^2\right\}\le C\exp(-c_2m)
  \end{equation}
  for fixed constants $C,c_2>0$ (which may depend on $t$). If $m\ge 3C(n_1\!+\!n_2\!+\!1)r$
  with $C>\frac{\log(36\sqrt{2}/\overline{\delta})}{c_2}$, then Assumption \ref{assump} holds
  for $s=r/2$ and $c=\sqrt{\frac{2\overline{\delta}}{1-\overline{\delta}}}$ with probability
  exceeding $1\!-\!2\exp(-dm)$ where $d\!=c_2-\frac{\log(36\sqrt{2}/\overline{\delta})}{C}$.
  Consequently, when $0\le\!\gamma_{k-1}<\!1/c$, the bounds in \eqref{noisyless-bound} holds
  with probability at least $1\!-\!2\exp(-dm)$ for such random measurements.
 \end{theorem}

  As remarked after \cite[Theorem 2.3]{Candes11}, the condition in \eqref{Acond} holds
  when $\mathcal{A}$ is a Gaussian random measurement ensemble (i.e., $A_1,\ldots,A_m$
  are independent from each other and each $A_i$ contains i.i.d. entries $\mathcal{N}(0,1/m)$);
  or when each entry of each $A_i$ has i.i.d. entries that are equally likely to take
  $\frac{1}{\!\sqrt{m}}$ or $-\frac{1}{\!\sqrt{m}}$; or when $\mathcal{A}$ is a random projection
  (see \cite{Recht10}).

%------------------------------------------------------------------------------------------------------Section 5
 \section{Numerical experiments}\label{sec5}

  In this section, we shall test the theoretical results in Section \ref{sec4} by applying Algorithm \ref{Alg1}
  to some low-rank matrix recovery problems, including matrix sensing and matrix completion problems.
  During the testing, we chose $\phi_2$ with $q=1/2$ and $\epsilon=10^{-3}$ for the function $\phi$
  in Algorithm \ref{Alg1}. Although Table \ref{tab1} shows that Algorithm \ref{Alg1} with $\phi_1$
  reduces the error faster than Algorithm \ref{Alg1} with $\phi_2$ does, our preliminary testing
  indicates that the latter has a little better performance in reducing the relative error.
  This accounts for choosing $\phi_2$ instead of $\phi_1$ for our numerical testing.
  In addition, we always chose $\rho_1=10/\|X^1\|$ and $\mu_k={5}/{4}\ (k\ge 1)$.
  All the results in the subsequent subsections were run on the Windows system
  with an Intel(R) Core(TM) i3-2120 CPU 3.30GHz.

%-------------------------------------------------------------------------------------------------Subsection 5.1

 \subsection{Low-rank matrix sensing problems}\label{subsec5.1}

  We tested the performance of Algorithm \ref{Alg1} with some matrix sensing problems,
  for which some entries are known exactly. Specifically, we assumed that ${\bf 5}$ entries of
  the true $\overline{X}\in\mathbb{R}^{n_1\times n_2}$ of rank $r$ are known exactly,
  and generated $\overline{X}$ in the following command
  \begin{verbatim}
         XR = randn(n1,r);  XL = randn(n2,r); Xbar = XR*XL'.
  \end{verbatim}
  \vspace{-0.5cm}
  We successively generated the matrices $A_1,\ldots,A_m\in\mathbb{R}^{n_1\times n_2}$ with i.i.d. standard
  normal entries to formulate the sampling operator $\mathcal{A}$. Such $\mathcal{A}$
  satisfies the RIP property with a high probability by \cite{Recht10}, which means that the restricted
  eigenvalues of $\mathcal{A}^*\mathcal{A}$ can satisfy Assumption \ref{assump} with a high probability
  from the discussions after Assumption \ref{assump}. Then, we successively generated the standard
  Gaussian noises $\xi_1,\ldots,\xi_m$ to formulate $b$ by
  \begin{equation}\label{observe}
     b=\mathcal{A}\overline{X}+0.1(\|\mathcal{A}\overline{X}\|/{\|\xi\|}){\xi}\ \ {\rm with}\ \
    \xi=(\xi_1,\ldots,\xi_m)^{\mathbb{T}}.
   \end{equation}
  Take $\delta=0.1\|b\|$ and $\Omega=\big\{X\in\mathbb{R}^{n_1\times n_2}\ |\ \mathcal{B}X=d,\,\|X\|\le R\big\}$
  for $R=2\|\overline{X}\|$, where $\mathcal{B}X:=(X_{ij})$ for $(i,j)\in\Upsilon_{\rm fix}$
  with $\Upsilon_{\rm fix}$ being the index set of known entries,
  and $d\in\mathbb{R}^{|\Upsilon_{\rm fix}|}$ is the vector
  consisting of $\overline{X}_{ij}$ for $(i,j)\in\Upsilon_{\rm fix}$.
  Let $\mathbb{I}_{S}(\cdot)$ denote the indicator function over a set $S$.
  The subproblem \eqref{subprob-X} in Algorithm \ref{Alg1} now has the form
  \begin{align}\label{primal1-subprob1}
   &\min_{X,Z\in\mathbb{R}^{n_1\times n_2},z\in\mathbb{R}^m}\!\|X\|_*-\langle C,X\rangle+\mathbb{I}_{\mathcal{R}}(z)+\mathbb{I}_{\Lambda}(Z)\nonumber\\
   &\qquad\quad {\rm s.t.}\ \ \mathcal{A}X-z-b=0,\,\mathcal{B}X-d=0,\,X-Z=0,
  \end{align}
  where $\mathcal{R}\!:=\{z\in\mathbb{R}^m\ |\ \|z\|\le\delta\}$ and
  $\Lambda\!:=\{X\in\mathbb{R}^{n_1\times n_2}\ |\ \|X\|\le R\}$.
  The dual of \eqref{primal1-subprob1} is
  \begin{align}\label{dual1-subprob1}
   &\min_{Y,\Gamma\in\mathbb{R}^{n_1\times n_2},\xi, u\in\mathbb{R}^m,\eta\in\mathbb{R}^{|\Upsilon_{\rm fix}|}}
  \langle b,\xi\rangle+\langle d,\eta\rangle+\delta\| u\|+ R \|Y\|_*\nonumber\\
   &\qquad\qquad\ {\rm s.t.}\ \ C-\mathcal{A}^*(\xi)-\mathcal{B}^*(\eta)-Y-\Gamma=0,\,\xi-u =0,\,\|\Gamma\|\le 1.
  \end{align}
  During the testing, we solved the subproblems of the form \eqref{primal1-subprob1}
  until the primal and dual relative infeasibility is less than $10^{-6}$ and the difference
  between the primal objective value and the dual one is less than $10^{-5}$,
  with the powerful Schur-complement based semi-proximal ADMM (alternating direction method of multipliers)
  \cite{LST161} for problem \eqref{dual1-subprob1}.
 %-------------------------------------------------------------------------------------------------Subsection 5.1

 \subsubsection{Performance of Algorithm \ref{Alg1} in different stages}\label{subsubsec5.1}

  We generated randomly a matrix sensing problem with some entries known as above with $n_1=n_2=100$,
  $r=6$ and $m=2328$ to test the performance of Algorithm \ref{Alg1} in different stages.
  Figure \ref{fig1} plots the relative error of Algorithm \ref{Alg1} in the first fifteen stages.
  We see that Algorithm \ref{Alg1} reduces the relative error of the nuclear norm relaxation method
  most in the second stage, and after the third stage the reduction becomes insignificant.
  This performance coincides with the analysis results shown as in Table \ref{tab1}.
 %----------------------------------------  Please insert Figure 5.1---------------------------------------------
 \begin{figure}[htbp]
   \setlength{\abovecaptionskip}{1pt}
 \begin{center}
 {\includegraphics[width=1.0\textwidth]{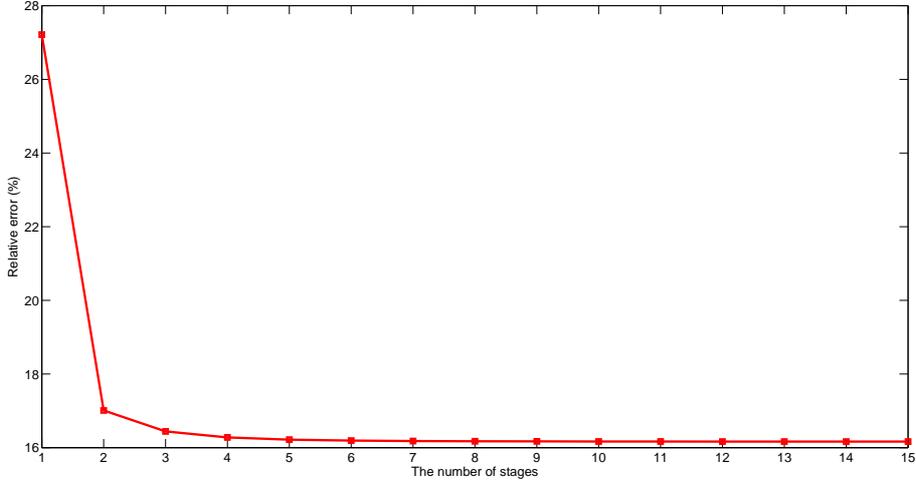}}\\
 \caption{\small Performance of Algorithm \ref{Alg1} in the first fifteen stages}
 \label{fig1}
 \end{center}
 \end{figure}

 %-------------------------------------------------------------------------------------------------Subsection 5.2
  \subsubsection{Performance of Algorithm \ref{Alg1} with different samples}\label{subsubsec5.2}

  We generated randomly a matrix sensing problem with some entries known as above with $n_1=n_2=100$,
  $r=5$ and $m=\nu r(2n\!-\!r)$ for $\nu\in\{1.0,1.1,\ldots,3.0\}$ to test the performance of
  Algorithm \ref{Alg1} under different samples. Figure \ref{fig2} depicts the relative error curves
  and the rank curves of the first stage convex relaxation and the first five stages convex relaxation,
  respectively. We see that the relative errors of the first stage convex relaxation
  and the first five stages convex relaxation decrease as the number of samples increases,
  but the relative error of the latter is always smaller than that of the former. Moreover,
  the first five stages convex relaxation reduces those of the first stage convex relaxation
  at least $25\%$ for $\nu\in[1.0,3.0]$, and the reduction becomes less as the number of samples increases.
  In particular, the rank of $X^1$ is higher than that of $\overline{X}$ even for $30\%$
  sampling ratio, but the rank of $X^5$ equals that of $\overline{X}$ even for $12\%$ sampling ratio.

   \begin{figure}[htbp]
   \setlength{\abovecaptionskip}{1pt}
   \begin{center}
   {\includegraphics[width=1.05\textwidth]{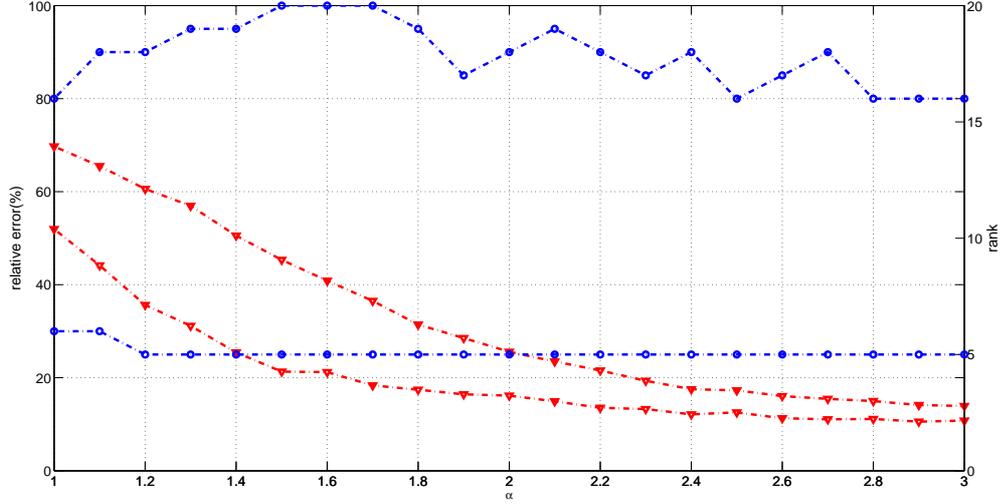}}
   \caption{Performance of the first stage and the first five stages convex relaxation}
    \label{fig2}
   \end{center}
   \end{figure}

%-------------------------------------------------------------------------------------------Subsection 5.3
 \subsection{Low-rank PSD matrix completion problems}\label{subsec5.3}

  We applied Algorithm \ref{Alg1} to two classes of low-rank PSD matrix
  completion problems. Although the sampling operators for such problems do not
  satisfy the RIP property, it is possible for the restricted eigenvalues
  of $\mathcal{A}^*\mathcal{A}$ to satisfy Assumption \ref{assump}. For these problems,
  \(
   \Omega=\big\{X\in\mathbb{S}_{+}^{n}\ |\ \mathcal{E}_1(X)=g_1,\,\mathcal{E}_2(X)\le g_2\big\}
  \)
  where $\mathcal{E}_1\!:\mathbb{S}^n\to\mathbb{R}^{l_1}$ and
  $\mathcal{E}_2\!:\mathbb{S}^n\to\mathbb{R}^{l_2}$ are the linear operators,
  and $g_1\in\mathbb{R}^{l_1}$ and $g_2\in\mathbb{R}^{l_2}$ are the given vectors.
  For this case, the subproblem \eqref{subprob-X} in Algorithm \ref{Alg1} now
  takes the form of
  \begin{align}\label{primal2-subprob1}
   &\min_{X\in\mathbb{S}^{n},z\in\mathbb{R}^m,y\in\mathbb{R}^{l_2}} \langle C,X\rangle+\mathbb{I}_{\mathbb{S}_{+}^n}(X)+\mathbb{I}_{\mathcal{R}}(z)+\mathbb{I}_{\mathbb{R}_{+}^{l_2}}(y)\nonumber\\
   &\qquad\ \ {\rm s.t.}\ \ \mathcal{A}X-z-b=0,\,
              \left(\begin{matrix}
              \mathcal{E}_1\\ \mathcal{E}_2
              \end{matrix}\right)X
              -\left(\begin{matrix}
                g_1\\ g_2
                \end{matrix}\right)+
                \left(\begin{matrix}
                0 \\ y
                \end{matrix}\right)=0.
  \end{align}
  After an elementary calculation, the dual problem of \eqref{primal2-subprob1} has the following form
  \begin{align}\label{dual2-subprob1}
   &\min_{\Gamma\in\mathbb{S}^n,\xi\in\mathbb{R}^m,\eta_1\in\mathbb{R}^{l_1},\eta_2,u\in\mathbb{R}^{l_2}}
   \langle b,\xi\rangle+\langle g_1,\eta_1\rangle+\langle g_2,\eta_2\rangle
   +\delta\| \xi\|+\mathbb{I}_{\mathbb{S}_{+}^n}(\Gamma)+\mathbb{I}_{\mathbb{R}_{+}^{l_2}}(u)\nonumber\\
   &\qquad\qquad\ {\rm s.t.}\ \ C+\mathcal{A}^*(\xi)+\mathcal{E}_1^*(\eta_1)+\mathcal{E}_2^*(\eta_2)-\Gamma=0,\,\eta_2-u=0.
  \end{align}
  During the testing, we solved the subproblems of the form \eqref{primal2-subprob1}
  until the primal and dual relative infeasibility is less than $10^{-6}$ and
  the difference between the primal objective value and the dual one is less than $10^{-5}$
  with the Schur-complement based semi-proximal ADMM \cite{LST161} for problem \eqref{dual2-subprob1},
  and stopped Algorithm \ref{Alg1} at the $k$th iteration once
  \[
    {\rm rank}(X^{k-2})\!=\!{\rm rank}(X^{k-1})\!=\!{\rm rank}(X^k),
  \]
  where ${\rm rank}(X^k)$ is the number of nonzero singular values of $X^k$
  less than $10^{-10}\cdot\sigma_{\rm max}(X^k)$.

%--------------------------------------------------------------------------------------------------
  \subsubsection{Low-rank correlation matrix completion problems}\label{subsubsec5.3.1}

  A correlation matrix is a real symmetric PSD matrix with all diagonals being $1$.
  We generated the true correlation matrix $\overline{X}\in\mathbb{S}_{+}^{n}$ of
  rank $r$ in the following command:
  \begin{verbatim}
    L = randn(n,r);  W = weight*L(:,1:1);  L(:,1:1) = W; G = L*L';
    M = diag(1./sqrt(diag(G)))*G*diag(1./sqrt(diag(G))); Xbar = (M+M')/2.
  \end{verbatim}
  \vspace{-0.5cm}
  In this way, one can control the ratio of the largest eigenvalue and the smallest
  nonzero eigenvalue of $\overline{X}$ by {\ttfamily weight}.
  We assume that some off-diagonal entries of $\overline{X}$ are known. Thus,
  $\mathcal{E}_1(X)=\left(\begin{matrix}
                           {\rm diag}(X)\\ \mathcal{B}(X)
                           \end{matrix}\right)$ for $X\in\mathbb{S}^n$,
  $g_1=\left(\begin{matrix}
   e \\ d
   \end{matrix}\right)$,
  $\mathcal{E}_2\equiv 0$ and $g_2=0$, where the operator $\mathcal{B}\!:\mathbb{S}^n\to\mathbb{R}^{|\Upsilon_{\rm fix}|}$
  and the vector $d\in\mathbb{R}^{|\Upsilon_{\rm fix}|}$ are defined as in Subsection \ref{subsec5.1}.
  The noise vector $\xi$ and the observation vector $b$ are generated in the same way as in \eqref{observe}.

  \medskip

  Table \ref{tab51} reports the numerical results of Algorithm \ref{Alg1} for some examples
  generated randomly. The information of $\overline{X}$ is reported
  in the first three columns, where the second column lists the number of known off-diagonal entries
  for $\overline{X}$, and the third column gives the ratio of the largest eigenvalue
  of $\overline{X}$ to the smallest nonzero eigenvalue of $\overline{X}$. For each test example,
  we sampled partial unknown off-diagonal entries uniformly at random to formulate the operator
  $\mathcal{A}$, where the sample ratio is ${\bf 1.92\%}$ for ${\rm rank}(\overline{X})=5$
  and ${\bf 4.32\%}$ for ${\rm rank}(\overline{X})=10$.
  The fourth and the fifth columns report the results of the first stage convex relaxation
  and the first two stages convex relaxation, respectively, and the sixth column reports
  the final result, where {\bf relerr(rank)} means the relative error and the rank of solutions,
  and {\bf iter} is the total number of iterations required by the Schur-complement
  based semi-proximal ADMM for the corresponding convex relaxation.
%-------------------------------------------------------------------------------------
  \begin{table}[htbp]
  \setlength{\abovecaptionskip}{2pt}
  \setlength{\belowcaptionskip}{0pt}
  \centering
   {\caption{\label{tab51} Performance for low-rank correlation matrix recovery problems with $n=1000$}}
   \begin{tabular}{|c|c|c|cc|cc|ccc|}
  \hline
   &    &  &\multicolumn{7}{c|}{Algorithm \ref{Alg1}}\\
   \cline{4-10}
  \raisebox{1.0ex}[0pt]{$r$}\!&\!{$\begin{array}{c} {\rm  off-} \\  {\rm diag} \end{array}$}\!&\!{\rm eigr}
  \!&\multicolumn{2}{c|}{\rm The first stage} &\multicolumn{2}{c|}{\rm The first two stages}&\multicolumn{3}{c|}{\rm Final result}\\
  \cline{4-10}
   &  &  &{\rm relerr(rank)}&\!{\rm iter}\!&{\rm relerr(rank)}&\!{\rm iter}\!& $k$ &{\rm relerr(rank)}&{\rm iter}\\
  \hline \hline
     &  0  &1.19  & 5.94e-1(1000)& 89 & 1.59e-1(6) & 1100 & 5 & 1.56e-1(5) & 4458\\

     &  0  &2.86  & 4.43e-1(1000)& 101 & 1.51e-1(5) & 629 & 4 & 1.52e-1(5) & 1805\\

     &  0  &4.36 & 3.57e-1(1000) & 98 & 1.49e-1(6) & 769 & 5 & 1.54e-1(5) & 2006\\

   \raisebox{1.5ex}[0pt]{5}  & 100 &1.17 & 5.81e-1(1000)& 88 & 1.53e-1(6)& 894 & 5 & 1.50e-1(5)& 4098\\

     & 100  &2.79  & 4.48e-1(1000) & 100 & 1.47e-1(5)& 697 & 4 & 1.48e-1(5) & 1734\\

     & 100  &4.23  & 3.60e-1(1000) & 98 & 1.48e-1(6) & 832 & 5 & 1.49e-1(5) & 2004\\

  \hline

     & 0  &1.36 & 4.16e-1(1000)& 67 & 1.48e-1(10) & 608 & 4 & 1.43e-1(10) & 1987\\

     & 0  &3.52 & 3.49e-1(1000) & 70 & 1.42e-1(10) & 464 & 4 & 1.41e-1(10) & 1152 \\

     & 0  &6.39 & 2.95e-1(1000) & 59 & 1.34e-1(10) & 373 & 4 & 1.37e-1(10) & 852\\

   \raisebox{1.5ex}[0pt]{10} & 100 &1.42 & 3.97e-1(1000)& 67 & 1.46e-1(10) & 562 & 4 & 1.42e-1(10) & 1934\\

     & 100  & 3.31 & 3.43e-1(1000)& 70 & 1.40e-1(10) & 470 & 4 & 1.40e-1(10) & 1218\\

     & 100  & 6.35 & 2.87e-1(1000)& 66 & 1.38e-1(10)& 344 & 5 & 1.42e-1(10)& 818\\
  \hline
  \end{tabular}
  \end{table}

  \medskip

  We see that the solution given by the trace-norm relaxation method has a high relative error
  and a full rank, while the two-stage convex relaxation reduces the relative error of
  the trace-norm relaxation at least $50\%$ for all test problems. Although the two-stage
  convex relaxation may yield the desirable relative error for all the test problems,
  the ranks of some problems (for example, the third and the fourth) are higher than that of $\overline{X}$.
  With the number of stages increasing, Algorithm \ref{Alg1} yields the same rank as that of $\overline{X}$.
  This indicates that for the problems with suitable sample ratios, the two-stage convex relaxation is enough;
  while for the problem with very low sample ratios, more than two stages convex relaxation is needed.
  In addition, since all the constraints to define the set $\Omega$ are of the hard type,
  some of the relative errors in the sixth column are little higher than those in the fifth column.
%--------------------------------------------------------------------------------------------------
  \subsubsection{Low-rank covariance matrix completion problems}\label{subsubsec5.3.1}

  We generated the true covariance matrix $\overline{X}\in\mathbb{S}_{+}^{n}$ of
  rank $r$ in the following command:
  \begin{verbatim}
       L = randn(n,r)/sqrt(sqrt(n));  W = weight*L(:,1:1);
       L(:,1:1) = W;  G = L*L'; Xbar = (G+G')/2.
  \end{verbatim}
  \vspace{-0.5cm}
  In this case, $\mathcal{E}_1=\mathcal{B}$ and $g_1=d$ where
  $\mathcal{B}\!:\mathbb{S}^n\to\mathbb{R}^{|\Upsilon_{\rm fix}|}$ and
  $d\in\mathbb{R}^{|\Upsilon_{\rm fix}|}$ are defined as in Subsection \ref{subsec5.1},
  $\mathcal{E}_2(X):=(X_{ii})$ for $(i,i)\in \Upsilon_{\rm diag}$  with $\Upsilon_{\rm diag}$
  being the index set of unknown diagonal entries of $\overline{X}$,
  and $g_2\in\mathbb{R}^{|\Upsilon_{\rm diag}|}$ is the vector consisting of the upper bounds
  for unknown diagonal entries of $\overline{X}$. We set
  $g_2=(1+0.01\textbf{rand}(1,1))\|\overline{X}\|_{\infty}\textbf{ones}(|\Upsilon_{\rm diag}|,1)$.

  \medskip

  Table \ref{tab52} reports the numerical results of Algorithm \ref{Alg1} for some problems generated randomly.
  The information of the true covariance matrix $\overline{X}$ is reported in the first two columns,
  where the second column lists the number of known diagonal and off-diagonal entries of $\overline{X}$,
  and the third column reports the ratio of the largest eigenvalue
  of $\overline{X}$ to the smallest nonzero eigenvalue of $\overline{X}$. For each test example,
  we sampled the upper triangular entries uniformly at random to formulate the sampling operator
  $\mathcal{A}$, where the sample ratio is ${\bf 1.91\%}$ for ${\rm rank}(\overline{X})=5$
  and ${\bf 5.72\%}$ for ${\rm rank}(\overline{X})=13$.
  The fourth and the fifth columns report the results of the first stage convex relaxation and
  the first two stages convex relaxation, respectively, and the last one lists the final results.
   \begin{table}[htbp]
  \centering
   {\caption{\label{tab52} Performance for low-rank covariance matrix recovery problems with $n=1000$}}
   \begin{tabular}{|c|c|c|cc|cc|ccc|}
  \hline
   &   &   &\multicolumn{7}{c|}{Algorithm \ref{Alg1}}\\
  \cline{4-10}
  \raisebox{1.5ex}[0pt]{$r$} & \!{$\begin{array}{c} ({\rm diag}, \\  {\rm offdiag}) \end{array}$}\!&\!{\rm eigr}\!
  &\multicolumn{2}{c|}{\rm The first stage}&\multicolumn{2}{c|}{\rm The first two stages} &\multicolumn{3}{c|}{\rm Final result}\\
  \cline{4-10}
   &    &  &\!{\rm relerr(rank)}\!& iter&\!{\rm relerr(rank)}\!& iter& $k$ &\!{\rm relerr(rank)}\!& iter \\
  \hline \hline
     & $(200,0)$ & 1.18 & 4.80e-1(36)& 787 & 2.42e-1(7)& 1309 & 4 & 2.25e-1(5)& 1996\\
    & $(200,0)$ & 4.59  &3.19e-1(32) & 676 & 1.92e-1(8)& 1214 & 4 & 1.82e-1(5)& 1800\\
 \raisebox{1.5ex}[0pt]{5}& $(0,200)$ & 1.20 & 4.86e-1(36)& 275 & 2.42e-1(7)& 451 & 4 & 2.21e-1(5)& 996\\
     & $(0,200)$ & 4.26 & 3.24e-1(33)& 349 & 1.98e-1(7) & 523 & 4 & 1.90e-1(5)& 1016\\
     & $(100,100)$ & 1.21& 4.74e-1(36)& 1038 & 2.42e-1(6) & 1734 & 4 &2.24e-1(5)& 2746\\
     & $(100,100)$ & 4.07 & 3.33e-1(33)& 879 & 1.92e-1(7) & 1354 & 4 & 1.80e-1(5)& 2155\\
  \hline
  & $(200,0)$ & 5.33 & 2.20e-1(53)& 253 & 1.56e-1(13)& 409 & 3 & 1.54e-1(13)& 554\\
    & $(200,0)$ & 7.72 & 1.90e-1(48) & 258 & 1.46e-1(13)& 393 & 3 & 1.45e-1(13)& 557\\
 \raisebox{1.5ex}[0pt]{13}& $(0,200)$ & 5.17 & 2.20e-1(53)& 177 & 1.58e-1(13)& 321 & 3 & 1.55e-1(13)& 466\\
     & $(0,200)$ & 9.01 & 1.78e-1(45)& 172 & 1.44e-1(13) & 323 & 3 & 1.43e-1(13)& 484\\
     & $(100,100)$ & 4.58 & 2.30e-1(54)& 256 & 1.59e-1(13) & 402 & 3 &1.56e-1(13)& 548\\
     & $(100,100)$ & 8.11 & 1.85e-1(48)& 221 & 1.47e-1(13) & 369& 3 & 1.46e-1(13)&530\\
  \hline
  \end{tabular}
  \end{table}

  \medskip

  We see that the solution yielded by the trace-norm relaxation method has a high relative
  error and rank, and the solution given by the first two stages convex relaxation
  has the desirable relative error but its rank is still higher than that of the true matrix
  for those problems with low sample ratios. This shows that for these difficult problems,
  more than two stages convex relaxation is required. For the problems with with $1.91\%$ sample ratio,
  the two-stage convex relaxation reduces the error bounds of the trace-norm relaxation method
  at least $38\%$, while for those problems with $5.32\%$ sample ratio, the reduction rate is over $19\%$.
  In addition, combining with the results in Table \ref{tab51}, we see that the performance
  of our multi-stage convex relaxation has no direct like with the ratio of the largest eigenvalue
  and the smallest nonzero eigenvalue of the true $\overline{X}$.

%------------------------------------------------------------------------------------------------------Section 7
  \section{Conclusions}

  We have proposed a multi-stage convex relaxation approach to the structured
  rank minimization problem (\ref{rank-min}) by solving the exact penalty problem
  of its equivalent MPGCC in an alternating way. It turned out that this approach
  not only has favorable theoretical guarantees but also reduces the error
  of the nuclear norm relaxation method effectively. There are several topics worthwhile
  to pursue, such as to develop the fast and effective algorithms for seeking
  the solution of subproblems, to establish the theoretical guarantee for the case
  where the subproblems are solved inexactly, and apply this approach to other classes
  of low-rank optimization problems, say, low-rank plus sparse problems.

  \bigskip
  \noindent
  {\large\bf Acknowledgements}\ \ This work is supported by the National Natural
 Science Foundation of China under project No.11571120 and the Natural Science Foundation of Guangdong Province under project No. 2015A030313214. The authors would like to thank Dr Xudong Li on help us improve our computing codes.

%----------------------------------------------------------------------------------------- References

  \bigskip
  \noindent
  {\bf\large Appendix A}

  \medskip

  Let $M\in\mathbb{R}^{n_1\times n_2}$ be a matrix of rank $\kappa>0$ with the SVD as
  \(
    U[{\rm Diag}(\sigma(M))\ \ 0] V^\mathbb{T},
  \)
  where $U\!=[U_1\ \ U_2]\in\mathbb{O}^{n_1}$ and $V=[V_1\ \ V_2]\in\mathbb{O}^{n_2}$
  with $U_1\in\mathbb{O}^{n_1\times\kappa}$ and $V_1\in\mathbb{O}^{n_2\times\kappa}$.
  Denote by $\mathcal {T}(M)$ the tangent space at $M$ associated to
  the rank constraint ${\rm rank}(X)\le \kappa$. Then, the subspace $\mathcal{T}(M)$
  and its orthogonal complementarity in $\mathbb{R}^{n_1\times n_2}$ have the form
  \begin{align}\label{Tangent-space}
   \mathcal{T}(M)=\big\{X\in\mathbb{R}^{n_1\times n_2}\ |\ X=U_1U_1^{\mathbb{T}}X+XV_1V_1^{\mathbb{T}}-U_1U_1^{\mathbb{T}}XV_1V_1^{\mathbb{T}}\big\},\\
   \mathcal{T}(M)^{\perp}=\big\{X\in\mathbb{R}^{n_1\times n_2}\ |\ X=U_2U_2^{\mathbb{T}}XV_2V_2^{\mathbb{T}}\big\}.\qquad\qquad\nonumber
  \end{align}
  In this part, we let $\widetilde{X} \in\mathbb{R}^{n_1\times n_2}$ be a matrix of rank $\kappa>0$
  with the SVD given by $\widetilde{U}[{\rm Diag}(\sigma(\widetilde{X}))\ \ 0]\widetilde{V}^{\mathbb{T}}$,
  where $\widetilde{U}=[\widetilde{U}_1\ \ \widetilde{U}_2]\in\mathbb{O}^{n_1}$ with
  $\widetilde{U}_1\in\mathbb{O}^{n_1\times\kappa}$ and
  $\widetilde{V}=[\widetilde{V}_1\ \ \widetilde{V}_2]\in\mathbb{O}^{n_2}$ with
  $\widetilde{V}_1\in\mathbb{O}^{n_2\times\kappa}$. We shall derive an upper bound
  for the projection of the perturbed $\widetilde{U}_{1}\widetilde{V}_{1}^\mathbb{T}$
  by a matrix $W\!\in\mathbb{R}^{n_1\times n_2}$ onto the subspaces $\mathcal {T}(\widetilde{X})$
  and $\mathcal {T}(\widetilde{X})^\perp$, respectively.
 %-------------------------------------------------------------------------------------------------------Lemma2.1
  \begin{alemma}\label{WU1V1}
   For any given $W\in\mathbb{R}^{n_1\times n_2}$ with the SVD as $U[{\rm Diag}(w_1,\ldots,w_{n_1})\ \ 0]V^\mathbb{T}$,
   where $U\!=[U_1\ \ U_2]\in\mathbb{O}^{n_1}$ with $U_1\in\mathbb{O}^{n_1\times\kappa}$ and
   $V\!=[V_1\ \ V_2]\in\mathbb{O}^{n_2}$ with $V_1\in\mathbb{O}^{n_2\times\kappa}$,
   \begin{equation}\label{WU1V1-ineq1}
   \|\mathcal{P}_{\mathcal {T}(\widetilde{X})^\perp}(W)\|
    \le w_{\kappa+1}+(w_1-w_{\kappa+1})\big\|\widetilde{U}_1\widetilde{V}_1^\mathbb{T}-U_1V_1^\mathbb{T}\big\|^2,
   \end{equation}
   \begin{equation}\label{WU1V1-ineq2}
   \|\mathcal {P}_{\mathcal {T}(\widetilde{X})}(\widetilde{U}_{1}\widetilde{V}_{1}^\mathbb{T}\!-\!W)\|_F
    \le\! (1\!+\!\sqrt{2}w_{\kappa+1})\|U_{1}V_{1}^\mathbb{T}\!-\!\widetilde{U}_1\widetilde{V}_1^\mathbb{T}\|_F
        \!+\!\sqrt{\kappa}\max\big(|1\!-\!w_{1}|,|1\!-\!w_\kappa|\big).
   \end{equation}
  \end{alemma}
  \begin{proof}
   Let $\Sigma_1\!:=\!{\rm Diag}(w_1,\ldots,w_\kappa)$ and $\Sigma_2\!:=\!{\rm Diag}(w_{\kappa+1},\ldots,w_{n_1})$.
   Then, we have that
  \begin{align}
   \big\|\mathcal{P}_{\mathcal {T}(\widetilde{X})^\perp}(W)\big\|
   &=\big\|\widetilde{U}_2\widetilde{U}_2^\mathbb{T}\big[U_1(\Sigma_1\!-\!w_{\kappa+1}I)V_1^\mathbb{T}
      +(w_{\kappa+1}U_1V_1^\mathbb{T}+U_2[\Sigma_2\ \ 0]V_2^\mathbb{T})\big]\widetilde{V}_2\widetilde{V}_2^\mathbb{T}\big\|\nonumber\\
   %&\le\big\|\widetilde{U}_2\widetilde{U}_2^\mathbb{T}U_1(\Sigma_1\!-\!w_{\kappa+1}I)V_1^\mathbb{T}\widetilde{V}_2\widetilde{V}_2^\mathbb{T}\big\|+w_{\kappa+1}\nonumber\\
   &\le\big\|\widetilde{U}_2\widetilde{U}_2^\mathbb{T}U_1\big\|\big\|\Sigma_1\!-\!w_{\kappa+1}I\big\|\big\|\widetilde{V}_2\widetilde{V}_2^\mathbb{T}V_1\big\|+w_{\kappa+1}\nonumber\\
   & = (w_{1}-w_{\kappa+1})\big\|\widetilde{U}_2\widetilde{U}_2^\mathbb{T}U_1V_1^\mathbb{T}\big\|
         \big\|\widetilde{V}_2\widetilde{V}_2^\mathbb{T}V_1U_1^\mathbb{T}\big\|+w_{\kappa+1}\nonumber\\
   &= (w_{1}-w_{\kappa+1})\big\|\widetilde{U}_2\widetilde{U}_2^\mathbb{T}(U_1V_1^\mathbb{T}-\widetilde{U}_1\widetilde{V}_1^\mathbb{T})\big\|
      \big\|\widetilde{V}_2\widetilde{V}_2^\mathbb{T}(V_1U_1^\mathbb{T}-\widetilde{V}_1\widetilde{U}_1^\mathbb{T})\big\|+w_{\kappa+1}\nonumber\\
   &\le (w_{1}-w_{\kappa+1})\big\|U_1V_1^\mathbb{T}-\widetilde{U}_1\widetilde{V}_1^\mathbb{T}\big\|^2+w_{\kappa+1},\nonumber
  \end{align}
  where the first inequality is using $\|w_{\kappa+1}U_1V_1^\mathbb{T}\!+\!U_2[\Sigma_2\ \ 0]V_2^\mathbb{T}\|\!\leq\! w_{\kappa+1}$,
  and the second equality is due to $\|Z\|=\|ZQ^\mathbb{T}\|$ for any $Z$ and $Q$
  with $Q^\mathbb{T}Q=I$. So, inequality (\ref{WU1V1-ineq1}) holds.
  In order to establish inequality (\ref{WU1V1-ineq2}),
  we first notice that for any $Z\in\mathbb{R}^{(n_1-\kappa)\times (n_2-\kappa)}$,
  \begin{align}
   \big\|\mathcal{P}_{\mathcal {T}(\widetilde{X})}(U_2ZV_2^\mathbb{T})\big\|_F
   %&= \big\|\widetilde{U}_1\widetilde{U}_1^\mathbb{T}U_2ZV_2^\mathbb{T}+U_2ZV_2^\mathbb{T}\widetilde{V}_1\widetilde{V}_1^\mathbb{T}
  %         -\widetilde{U}_1\widetilde{U}_1^\mathbb{T}U_2ZV_2^\mathbb{T}\widetilde{V}_1\widetilde{V}_1^\mathbb{T}\big\|_F\nonumber\\
 % &=\big\|\widetilde{U}_1\widetilde{U}_1^\mathbb{T}U_2ZV_2^\mathbb{T} +\widetilde{U}_2\widetilde{U}_2^\mathbb{T}U_2ZV_2^\mathbb{T}\widetilde{V}_1\widetilde{V}_1^\mathbb{T}\big\|_F\nonumber\\
   &=\sqrt{\big\|\widetilde{U}_1\widetilde{U}_1^\mathbb{T}U_2ZV_2^\mathbb{T}\big\|_F^2
       +\big\|\widetilde{U}_2\widetilde{U}_2^\mathbb{T}U_2ZV_2^\mathbb{T}\widetilde{V}_1\widetilde{V}_1^\mathbb{T}\big\|_F^2}\nonumber\\
   & \le  \sqrt{\|Z\|^2\big\|\widetilde{U}_1^\mathbb{T}U_2\big\|_F^2
       +\|Z\|^2\big\|V_2^\mathbb{T}\widetilde{V}_1\big\|_F^2}\nonumber\\
   & = \|Z\|\sqrt{\big\|(\widetilde{V}_1\widetilde{U}_1^\mathbb{T}-V_1U_1^\mathbb{T})U_2\big\|_F^2
       +\big\|V_2^\mathbb{T}(\widetilde{V}_1\widetilde{U}_1^\mathbb{T}-V_1U_1^\mathbb{T})\big\|_F^2}\nonumber\\
   &\le \sqrt{2}\|Z\|\|\widetilde{V}_1\widetilde{U}_1^\mathbb{T}-V_1U_1^\mathbb{T}\|_F,\nonumber
  \end{align}
  where the first equality is by the expression of $\mathcal{P}_{\mathcal{T}(\widetilde{X})}(\cdot)$.
  Then, it holds that
  \begin{align}
   \big\|\mathcal {P}_{\mathcal {T}(\widetilde{X})}(W\!-\widetilde{U}_{1}\widetilde{V}_{1}^\mathbb{T})\big\|_F
   &\le \big\|\mathcal {P}_{\mathcal {T}(\widetilde{X})}(\widetilde{U}_{1}\widetilde{V}_{1}^\mathbb{T}\!-\!U_1\Sigma_1V_1^\mathbb{T})\big\|_F
        +\big\|\mathcal {P}_{\mathcal {T}(\widetilde{X})}(U_2[\Sigma_2\ \ 0]V_2^\mathbb{T})\big\|_F\nonumber\\
   &\le \big\|\widetilde{U}_{1}\widetilde{V}_{1}^\mathbb{T}\!-U_1\Sigma_1V_1^\mathbb{T}\big\|_F
        + \sqrt{2}\|[\Sigma_2\ \ 0]\|\|\widetilde{V}_1\widetilde{U}_1^\mathbb{T}-V_1U_1^\mathbb{T}\|_F\nonumber\\
   & \le (1+\sqrt{2}w_{\kappa+1})\big\|U_{1}V_{1}^\mathbb{T}\!-\widetilde{U}_1\widetilde{V}_1^\mathbb{T}\big\|_F+\big\|U_1(I-\Sigma_1)V_1^\mathbb{T}\big\|_F
     \nonumber\\
   & \le (1+\sqrt{2}w_{\kappa+1})\|U_{1}V_{1}^\mathbb{T}\!-\!\widetilde{U}_1\widetilde{V}_1^\mathbb{T}\|_F
        +\sqrt{\kappa}\max\big(|1\!-\!w_{1}|,|1\!-\!w_\kappa|\big).\nonumber
  \end{align}
  This shows that inequality (\ref{WU1V1-ineq2}) holds. Thus, we complete the proof.
  \end{proof}

  When the matrix $W$ in Lemma \ref{WU1V1} has the simultaneous SVD as a matrix $X$ close to $\widetilde{X}$,
  by \cite[Theorem 3]{MiaoPS16} the term $\|\widetilde{U}_1\widetilde{V}_1^\mathbb{T}\!-\!U_1V_1^\mathbb{T}\|$
  can be upper bounded as follows.
%-------------------------------------------------------------------------------------------------------Lemma
  \begin{alemma}\label{lemma1-A}
   Let $X\in\mathbb{R}^{n_1\times n_2}$ be an arbitrary matrix of rank $\kappa>0$ with
   the SVD given by $U[{\rm Diag}(\sigma(X))\ \ 0]V^\mathbb{T}$, where $U\!=[U_1\ \ U_2]\in\mathbb{O}^{n_1}$
   with $U_1\in\mathbb{O}^{n_1\times\kappa}$ and  $V\!=[V_1\ \ V_2]\in\mathbb{O}^{n_2}$ with $V_1\in\mathbb{O}^{n_2\times\kappa}$.
   For any given $\omega>2$, if $\|X\!-\!\widetilde{X}\|_F\le\eta$ for some
   $\eta\in\big(0,\frac{\sigma_{\kappa}(\widetilde{X})}{\omega}\big]$, then it holds that
   \(
      \big\|U_1V_1^{\mathbb{T}}-\widetilde{U}_{1}\widetilde{V}_{1}^{\mathbb{T}}\big\|_F
    \le \frac{1}{\sqrt{2}}\ln\big(\frac{\omega}{\omega-\!\sqrt{2}}\big).
   \)
  \end{alemma}

  \vspace{0.3cm}
  \noindent
  {\bf\large Appendix B}

  \medskip

  This part includes two results on the restricted eigenvalues of $\mathcal{A}^*\mathcal{A}$.
  The first gives a relation among $\vartheta_{+}(\cdot),\vartheta_{-}(\cdot)$
  and $\pi(\cdot,\cdot)$ where for given positive integers $k,l$ with $k+l\!\le n_1$,
  \begin{equation}\label{pikl}
  \pi(k,l):=\sup_{0<{\rm rank}(X)\leq k,\atop 0<{\rm rank}(Y)\leq l,\langle X,Y\rangle=0}
  \frac{\langle X,\mathcal {A}^*\mathcal {A}(Y)\rangle\|X\|_F}{\|\mathcal {A}(X)\|^2\|Y\|}.
  \end{equation}
  \vspace{-0.3cm}
 %----------------------------------------------------------------------------------Lemma
  \begin{alemma}\label{lemma1-reig}
   For any given positive integer $k,l$ with $k+l\le n_1$,
   \(
     \pi(k,l)\le \frac{\sqrt{l}}{2}\sqrt{\frac{\vartheta_{+}(l)}{\vartheta_{-}(k+l)}-1}.
   \)
  \end{alemma}

  Since the proof of Lemma \ref{lemma1-reig} is similar to that of \cite[Proposition 3.1]{Zhang09},
  we omit it. The second one is an extension of \cite[Lemma 10.1]{Zhang09} in the matrix setting,
  stated as follows.
 %----------------------------------------------------------------------------------------------------Lemma 2.4
  \begin{alemma}\label{lemma2-reig}
   Let $G\in \mathbb{R}^{n_1\times n_2}, U_{\!J}\in\mathbb{O}^{n_1\times |J|}$ and
   $V_{J'}\in\mathbb{O}^{n_2\times |J'|}$ be given. Let $U_{\!J}^{\mathbb{T}}GV_{J'}$
   have the SVD as $P[{\rm Diag}\big(\sigma(U_J^{\mathbb{T}}GV_{J'})\big)\ \ 0] Q^{\mathbb{T}}$,
   where $P=[P_1\ \ P_2]\in\mathbb{O}^{|J|}$ and $Q=[Q_1\ \ Q_2]\in\mathbb{O}^{|J'|}$
   with $P_1\in\mathbb{O}^{|J|\times s}$ and $Q_1\in\mathbb{O}^{|J'|\times s}$
   for an integer $1\leq s\leq \min(|J|,|J'| )$. Let
  $\mathcal{G}=\mathcal{L}^{\perp}\oplus\mathcal{J}_1$ with
  \(
    \mathcal{L}\!:=\!\big\{U_{\!J}ZV_{J'}^\mathbb{T}\ |\ Z\in \mathbb{R}^{|J|\times |J'|}\big\}
  \)
  and
  \(
    \mathcal{J}_1\!:=\!\big\{U_{\!J}P_1Z(V_{J'}Q_{1})^{\mathbb{T}}\ |\ Z\in\mathbb{R}^{s\times s }\big\}.
  \)
  Then, for any $H\in\mathcal{G}$, the following inequality holds with $l=\max_{Z\in \mathcal{L}^\perp}{\rm rank}(Z)$:
  \begin{align}
    \max\big(0,\langle H,\mathcal {A}^*\mathcal {A}(G)\rangle\big)
    &\ge \vartheta_-( l+s)\big(\|H\|_F -s^{-1}\pi( l+s,s)\big\|\mathcal{P}_{\mathcal{L}}(G)\big\|_*\big)\|H\|_F\nonumber\\
    &\quad -\vartheta_+( l+s)\|H\|_F\|\mathcal{P}_{\mathcal{G}}(G-H)\|_F.\nonumber
  \end{align}
  \end{alemma}
  \begin{proof}
   Let $H$ be an arbitrary matrix from $\mathcal{G}$.
   If $\|H\|_F\!\le\! \frac{\pi( l+s,s)}{s}\big\|\mathcal{P}_{\mathcal{L}}(G)\big\|_*$,
   the conclusion is clear. So, we assume that
   $\|H\|_F\!>\! \frac{\pi( l+s,s)}{s}\big\|\mathcal{P}_{\mathcal{L}}(G)\big\|_*$.
   By the definition of $\vartheta_+(l+s)$,
   $\|\mathcal{A}\mathcal{P}_{\mathcal{G}}(H-G)\|^2 \le\vartheta_+( l+s)\|\mathcal{P}_{\mathcal{G}}(H-G)\|_F^2$
   and $\|\mathcal{A}(H)\|^2\le\vartheta_+( l+s)\|H\|_F^2$. Then,
   \begin{equation}\label{equa-preprop}
     \big\langle H,\mathcal{A}^*\mathcal{A}\mathcal{P}_{\mathcal{G}}(G-H)\big\rangle
     \ge -\|\mathcal{A}(H)\|\|\mathcal{A}\mathcal{P}_{\mathcal{G}}(H-G)\|
     \ge -\vartheta_+( l+s)\|H\|_F\|\mathcal{P}_{\mathcal{G}}(H-G)\|_F.
   \end{equation}
   We proceed the arguments by considering the following two cases.

   \medskip
   \noindent
  {\bf Case 1:} ${\rm rank}(U_{\!J}^{\mathbb{T}}GV_{\!J'})\leq s\leq \min(|J|,|J'| )$.
  Now, by the expression of $\mathcal {P}_{\mathcal{J}_1}$, we have
  \[
   \mathcal{P}_{\mathcal{J}_1}(G)=U_{\!J}P_{1}P_{1}^{\mathbb{T}}U_{\!J}^{\mathbb{T}}GV_{\!J'}Q_{1}Q_{1}^{\mathbb{T}}V_{\!J'}^{\mathbb{T}}
   =U_{\!J}P_{1}\big[{\rm Diag}(\sigma(U_{\!J}^{\mathbb{T}}GV_{\!J'}))\ \ 0\big]Q_{1}^{\mathbb{T}}V_{\!J'}^{\mathbb{T}}
   =U_{\!J}U_{\!J}^{\mathbb{T}}GV_{\!J'}V_{\!J'}^{\mathbb{T}},
  \]
  where the last two equalities are due to $U_{\!J}^{\mathbb{T}}GV_{\!J'}
  \!=\!P_{1}[{\rm Diag}\big(\sigma(U_{\!J}^{\mathbb{T}}GV_{\!J'})\big)\ \ 0]Q_{1}^{\mathbb{T}}$.
  Note that $\mathcal {P}_{\mathcal{L}}(G)=U_{\!J}U_{\!J}^{\mathbb{T}}GV_{\!J'}V_{\!J'}^{\mathbb{T}}$
  by the definition of $\mathcal{L}$. So, $\mathcal{P}_{\mathcal {L}}(G)=\mathcal{P}_{\mathcal {J}_1}(G)$,
  i.e., $G\in \mathcal{G}$. Then,
  \begin{align}
   \langle \mathcal{A}(H),\mathcal{A}(G)\rangle
%   &=\langle \mathcal{A}(H),\mathcal{A}(H)\rangle + \langle \mathcal{A}(H),\mathcal{A}(G-H)\rangle\nonumber\\
   &=\langle \mathcal{A}(H),\mathcal{A}(H)\rangle + \langle \mathcal{A}(H),\mathcal{A}\mathcal{P}_{\mathcal {G}}(G-H)\rangle\nonumber\\
   &\geq \vartheta_-( l+s)\|H\|_F^2-\vartheta_+( l+s)\|H\|_F\|\mathcal{P}_{\mathcal{G}}(H-G)\|_F.\nonumber
  \end{align}
  This inequality implies the desired result. Thus, we complete the proof for this case.

  \medskip
   \noindent
  {\bf Case 2:} $s<{\rm rank}(U_{\!J}^{\mathbb{T}}GV_{\!J'})$.
  Let $k$ be the smallest positive integer such that $sk\geq \min(|J|,|J'|)$.
  Clearly, $k\ge 2$. Let $l_{i}$ and $\widetilde{l}_{i}$ for $i=1,2,\ldots,k$ be such that
  \[
    l_{1}=\cdots=l_{k-1}=s,\ l_{k}=|J|-s(k-1),\ \ \widetilde{l}_{1}=\cdots=\widetilde{l}_{k-1}=s,\ \widetilde{l}_{k}=|J'|-s(k-1).
  \]
  For each $2\leq i\leq k$, we define the subspace
  $\mathcal {J}_i:=\big\{U_{\!J}P_iZ(V_{\!J}Q_{i})^{\mathbb{T}}\ |\ Z\in\mathbb{R}^{l_{i}\times \widetilde{l}_{i} }\big\}$,
   where $P_i\in\mathbb{O}^{|J|\times l_{i}}$ is the matrix consisting of the $(\sum_{j=1}^{i-1}l_{j}\!+\!1)$th column
   to the $(\sum_{j=1}^{i}l_{j})$th  column of $P$; and $Q_i\in\mathbb{O}^{|J|\times \widetilde{l}_{i}}$ is the matrix
   consisting of the $(\sum_{j=1}^{i-1}\widetilde{l}_{j}\!+\!1)$th column to the $(\sum_{j=1}^{i}\widetilde{l}_{j})$th  column of $Q$.
   Clearly, $\mathcal{J}_1\perp\mathcal{J}_i$ for $i\ge 2$. From the definition of $\mathcal{G}$,
   we have
   \(
     \mathcal{G}\perp\mathcal{J}_i
   \)
   for $i\ne 1$. For each $i\geq 1$, it is easy to calculate that
   \[
     \mathcal{P}_{\mathcal{J}_i}(Z)=U_{\!J}P_{i}(U_{\!J}P_{i})^{\mathbb{T}}ZV_{\!J'}Q_{i}(V_{\!J'}Q_{i})^{\mathbb{T}}
     \quad\ \forall Z\in\mathbb{R}^{n_1\times n_2}.
   \]
   This, together with $\mathcal{P}_{\mathcal {L}}(G)=U_{\!J}U_{\!J}^{\mathbb{T}}GV_{\!J'}V_{\!J'}^{\mathbb{T}}$, implies that
   \(
    \mathcal{P}_{\mathcal {L}}(G) =\sum_{i=1}^k\mathcal{P}_{J_i}(G).
   \)
   Then, $\langle H,\mathcal{A}^*\mathcal{A}(G)\rangle=\langle H,\mathcal{A}^*\mathcal{A}\mathcal{P}_{\mathcal{G}}(G)\rangle
    +{\textstyle{\sum_{i>1}\big\langle H,\mathcal {A}^*\mathcal {A}\mathcal{P}_{J_i}(G)\big\rangle}}$.
   Consequently, we have that
   \begin{align}\label{equa-preprop1}
     &\big\langle H,\mathcal{A}^*\mathcal{A}(G)\big\rangle-\big\langle H,\mathcal{A}^*\mathcal{A}\mathcal{P}_{\mathcal {G}}(G-H)\big\rangle\nonumber\\
     &=\langle H,\mathcal{A}^*\mathcal{A}(H)\rangle
       +\sum_{i>1}\big\langle\mathcal{P}_{\mathcal {G}}(H),\mathcal{A}^*\mathcal{A}\mathcal{P}_{\mathcal {J}_i}(G)\big\rangle\nonumber\\
     &=\langle H,\mathcal{A}^*\mathcal{A}(H)\rangle
       \left[1+\sum_{i>1}\frac{\langle H,\mathcal{A}^*\mathcal{A}\mathcal{P}_{\mathcal {J}_i}(G)\rangle\|H\|_F}{\|\mathcal{A}(H)\|^2\|\mathcal{P}_{\mathcal {J}_i}(G)\|}\frac{\|\mathcal{P}_{\mathcal {J}_i}(G)\|}{\|H\|_F}\right]\nonumber\\
     &\ge \langle H,\mathcal{A}^*\mathcal{A}(H)\rangle\Big[1-\pi( l+s,s)\frac{\sum_{i>1}\|\mathcal{P}_{\mathcal {J}_i}(G)\|}{\|H\|_F}\Big]\qquad\nonumber\\
     &\ge\langle H,\mathcal{A}^*\mathcal{A}(H)\rangle\Big[1-\frac{\pi( l+s,s)\|\mathcal{P}_{\mathcal{L}}(G)\|_*}{s\|H\|_F}\Big]\qquad\nonumber\\
     &\ge \vartheta_-( l+s)\|H\|_F\Big[\|H\|_F-s^{-1}\pi( l+s,s)\|\mathcal{P}_{\mathcal{L}}(G)\|_*\Big],
   \end{align}
   where the first inequality is using the definition of $\pi$ by the fact that
   $H\in\mathcal{G},\mathcal{P}_{\mathcal {J}_i}(G)\in\mathcal{J}_i$ and ${\rm rank}(\mathcal{P}_{\mathcal {J}_i}(G))\leq s$,
   $\mathcal{G}\perp\mathcal{J}_i$ for $i>1$, and the second inequality is due to
   \[
     {\textstyle{\sum_{i>1}\|\mathcal{P}_{\mathcal {J}_i}(G)\|
      \leq s^{-1}\sum_{i=1}\big\|\mathcal{P}_{\mathcal {J}_i}(G)\big\|_*
      =s^{-1}\|\mathcal{P}_{\mathcal{L}}(G)\|_*}},
   \]
   since $\|\mathcal{P}_{\mathcal {J}_{i+1}}(G)\|\le s^{-1}\|\mathcal{P}_{\mathcal {J}_i}(G)\|_*$.
   Combining (\ref{equa-preprop1}) with (\ref{equa-preprop}), we get the result.
  \end{proof}

  \vspace{0.3cm}
  \noindent
  {\bf\large Appendix C}

  \medskip

  This part includes the proofs of all the results in Section \ref{sec4}.
  For convenience, in this part we write $\Delta^{\!k}\!:=X^{k}\!-\!\overline{X}$ for $k\ge 1$.
  We first establish two preliminary lemmas.
%------------------------------------------------------------------------------------------------Lemma 5.1
  \begin{alemma}\label{lemma1-sec41}
   If $\|\mathcal{P}_{\mathcal{T}^{\perp}}(W^{k-1})\|<1$ for some $k\ge 1$, then with $\gamma_{k-1}$
   defined by \eqref{gammak+g}
   \[
    \|\mathcal{P}_{\mathcal{T}^{\perp}}(\Delta^{\!k})\|_*\leq\gamma_{k-1}\sqrt{2r}\|\mathcal{P}_{\mathcal {T}}(\Delta^{\!k})\|_F.
   \]
  \end{alemma}
  \begin{proof}
  By the optimality of $X^{k}$ and the feasibility of $\overline{X}$ to the subproblem \eqref{subprob-X},
  \[
   \|X^k\|_*-\langle W^{k-1},X^k\rangle\leq \|\overline{X}\|_*-\langle W^{k-1},\overline{X}\rangle.
  \]
  From the directional derivative of the nuclear norm at $\overline{X}$, it follows that
  \[
   \|X^k\|_*-\|\overline{X}\|_*\geq \langle \overline{U}_1\overline{V}_1^\mathbb{T}, X^k-\overline{X}\rangle
   + \|\mathcal {P}_{\mathcal{T}^\perp}(X^k-\overline{X})\|_*.
  \]
  The last two equations imply that
  \(
   \langle \overline{U}_1\overline{V}_1^\mathbb{T}, \Delta^{\!k}\rangle+\|\mathcal{P}_{\mathcal {T}^\perp}(\Delta^{\!k})\|_*
   \leq \langle W^{k-1}, \Delta^{\!k}\rangle.
  \)
  Hence,
  \[
  \langle \overline{U}_1\overline{V}_1^\mathbb{T}, \mathcal{P}_{\mathcal{T}}(\Delta^{\!k})\rangle+\|\mathcal{P}_{\mathcal{T}^{\perp}}(\Delta^{\!k})\|_*
  \le \langle W^{k-1}, \Delta^{\!k}\rangle.
  \]
  This, along with $\langle W^{k-1}, \Delta^{\!k}\rangle=\langle \mathcal{P}_{\mathcal{T}^{\perp}}(W^{k-1}),\mathcal{P}_{\mathcal{T}^{\perp}}(\Delta^{\!k})\rangle+\langle W^{k-1},  \mathcal{P}_{\mathcal {T}}(\Delta^{\!k})\rangle$, yields that
  \begin{align}
  \|\mathcal{P}_{\mathcal{T}^{\perp}}(\Delta^{\!k})\|_*-\langle \mathcal{P}_{\mathcal{T}^{\perp}}(W^{k-1}), \mathcal{P}_{\mathcal{T}^{\perp}}(\Delta^{\!k})\rangle
  &\le \langle \mathcal{P}_{\mathcal {T}}(W^{k-1}\!-\!\overline{U}_1\overline{V}_1^\mathbb{T}), \mathcal{P}_{\mathcal {T}}(\Delta^{\!k})\rangle.\nonumber
 \end{align}
  Using the relation $|\langle Y, Z\rangle|\leq \|Y\|\|Z\|_*$ for any $Y, Z\in \mathbb{R}^{n_1\times n_2}$,
 we obtain that
 \[
   \big(1\!-\|\mathcal{P}_{\mathcal{T}^{\perp}}(W^{k-1})\|\big)\|\mathcal{P}_{\mathcal{T}^{\perp}}(\Delta^{\!k})\|_*
   \le \|\mathcal{P}_{\mathcal {T}}(W^{k-1}\!-\!\overline{U}_1\overline{V}_1^\mathbb{T})\|_F\|\mathcal{P}_{\mathcal {T}}(\Delta^{\!k})\|_F.
 \]
 From this inequality and the definition of $\gamma_{k-1}$, we obtain the desired result.
 \end{proof}
 %-------------------------------------------------------------------------------------------------------------Lemma 5.2
 \begin{alemma}\label{lemma2-sec41}
  Suppose that $\|\mathcal{P}_{\mathcal {T}^{\perp}}(W^{k-1})\|<1$ for some $k\ge 1$.
  Let $\overline{U}_{\!2}^{\mathbb{T}}\Delta^{\!k}\overline{V}_{\!2}$ have the SVD as
  $P^k\big[{\rm Diag}(\sigma(\overline{U}_{\!2}^{\mathbb{T}}\Delta^{\!k}\overline{V}_{\!2}))\ \ 0\big](Q^k)^{\mathbb{T}}$
  where $P^k\!=[P_1^k\ \ P_2^k]\in\mathbb{O}^{n_1-r}$ and $Q^k\!=[Q_1^k\ \ Q_2^k]\in\mathbb{O}^{n_2-r}$
  with $P_{1}^k\in\mathbb{O}^{(n_1-r)\times s}$ and $Q_1^k\in\mathbb{O}^{(n_2-r)\times s}$
  for an integer $1\le s\le n_1\!-\!r$, and define $\mathcal{M}^k:=\mathcal {T}\oplus\mathcal{H}^k$ with
  $\mathcal {H}^k\!=\big\{\overline{U}_{\!2}P_1^kY(\overline{V}_{\!2}Q_1^k)^{\mathbb{T}}\ |\ Y\in\mathbb{R}^{s\times s}\big\}$.
  Then, it holds that
  \[
    \big\|\Delta^{\!k}\big\|_F\le \sqrt{1+r\gamma_{k-1}^2/(2s)}\,\big\|\mathcal{P}_{\mathcal {M}^k}(\Delta^{\!k})\big\|_F .
  \]
  \end{alemma}
  \begin{proof}
   By the definitions of the subspaces $\mathcal{T}^{\perp}$ and $\mathcal{H}^k$, for any $Z\in\mathbb{R}^{n_1\times n_2}$,
   \[
     \mathcal{P}_{\mathcal{T}^{\perp}}(Z)=\overline{U}_2\overline{U}_2^{\mathbb{T}}Z\overline{V}_{2}\overline{V}_2^{\mathbb{T}}
     \ \ {\rm and}\ \
     \mathcal{P}_{\mathcal{H}^k}(Z)=\overline{U}_2P_1^k(\overline{U}_2P_1^k)^{\mathbb{T}}Z\overline{V}_2Q_1^k(\overline{V}_2Q_1^k)^{\mathbb{T}}.
   \]
   This implies that $\mathcal{P}_{\mathcal{H}^k}(Z)\!=\!\mathcal{P}_{\mathcal{H}^k}(\mathcal{P}_{\mathcal{T}^{\perp}}(Z))$
   for $Z\!\in\!\mathbb{R}^{n_1\times n_2}$.
   So,
   \(
   \mathcal{P}_{\mathcal {H}^k}(\Delta^{\!k})\!=\!\mathcal{P}_{\mathcal {H}^k}(\mathcal{P}_{\mathcal{T}^{\perp}}(\Delta^{\!k})).
   \)
   In addition, by the expression of $\mathcal{P}_{\mathcal{T}^{\perp}}(\Delta^{\!k})$ and the SVD of
   $\overline{U}_2^{\mathbb{T}}\Delta^{\!k}\overline{V}_2$, we have that
    \begin{align}\label{equa51}
     \mathcal{P}_{\mathcal {T}^\perp}(\Delta^{\!k})=\overline{U}_2\big(\overline{U}_2^{\mathbb{T}}\Delta^{\!k}\overline{V}_2\big)\overline{V}_2^{\mathbb{T}}
     \!=\overline{U}_2P\big[{\rm Diag}(\sigma(\overline{U}_2^{\mathbb{T}}\Delta^{\!k}\overline{V}_2))\ \ 0\big]Q^{\mathbb{T}}\overline{V}_2^{\mathbb{T}}.
   \end{align}
   Thus, from the expression of $\mathcal{P}_{\mathcal{H}^k}(\mathcal{P}_{\mathcal{T}^{\perp}}(\Delta^{\!k}))$
   and equation (\ref{equa51}), it follows that
   \begin{align}\label{equa52}
    \mathcal{P}_{\mathcal {H}^k}(\Delta^{\!k})
    = \mathcal{P}_{\mathcal{H}^k}(\mathcal{P}_{\mathcal{T}^{\perp}}(\Delta^{\!k}))
    = \overline{U}_2P_1^k\big[{\rm Diag}\big(\sigma^{s,\downarrow}(\overline{U}_2^{\mathbb{T}}\Delta^{\!k}\overline{V}_2)\big)\ \ 0\big](Q_1^k)^{\mathbb{T}}\overline{V}_2^{\mathbb{T}},
   \end{align}
   where $\sigma^{s,\downarrow}(\overline{U}_{\!2}^{\mathbb{T}}\Delta^{\!k}\overline{V}_{\!2})$
   is the vector consisting of the first $s$ components of $\sigma(\overline{U}_{\!2}^{\mathbb{T}}\Delta^{\!k}\overline{V}_{\!2})$.
   Notice that $\mathcal{P}_{\mathcal {M}^k}(\Delta^{\!k})=\mathcal{P}_{\mathcal {T}}(\Delta^{\!k})+\mathcal{P}_{\mathcal {H}^k}(\Delta^{\!k})$
   since the subspaces $\mathcal {T}$ and $\mathcal {H}^k$ are orthogonal.
   By combining this with equalities (\ref{equa51}) and (\ref{equa52}), we can obtain that
   \begin{align*}\label{temp-equa1-sec5}
   &\|\Delta^{\!k}\!-\!\mathcal{P}_{\mathcal {M}^k}(\Delta^{\!k})\|
   = \|\mathcal{P}_{\mathcal{T}^{\perp}}(\Delta^{\!k})\!-\!\mathcal{P}_{\mathcal {H}^k}(\Delta^{\!k})\|
  \le s^{-1}\|\mathcal{P}_{\mathcal {H}^k}(\Delta^{\!k})\|_*, \nonumber\\
  &\|\Delta^{\!k}\!-\!\mathcal{P}_{\mathcal {M}^k}(\Delta^{\!k})\|_*
   = \|\mathcal{P}_{\mathcal{T}^{\perp}}(\Delta^{\!k})\!-\!\mathcal{P}_{\mathcal {H}^k}(\Delta^{\!k})\|_*
   = \|\mathcal{P}_{\mathcal{T}^{\perp}}(\Delta^{\!k})\|_*\!-\!\|\mathcal{P}_{\mathcal {H}^k}(\Delta^{\!k})\|_*.
  \end{align*}
  Together with $\|\mathcal{P}_{(\mathcal {M}^k)^\perp}(\Delta^{\!k})\|_F^2
  \le\|\mathcal{P}_{(\mathcal {M}^k)^\perp}(\Delta^{\!k})\|\|\mathcal{P}_{(\mathcal {M}^k)^\perp}(\Delta^{\!k})\|_*$
  and Lemma \ref{lemma1-sec41},
  \begin{align}
  \|\mathcal{P}_{(\mathcal {M}^k)^\perp}(\Delta^{\!k})\|_F
  & \leq \big(\|\Delta^{\!k}\!-\!\mathcal{P}_{\mathcal {M}^k}(\Delta^{\!k})\|\|\Delta^{\!k}\!-\!\mathcal{P}_{\mathcal {M}^k}(\Delta^{\!k})\|_*\big)^{1/2}\leq \frac{1}{2\sqrt{s}}\|\mathcal{P}_{\mathcal{T}^{\perp}}(\Delta^{\!k})\|_*\nonumber\\
  &\leq \frac{\gamma_{k-1}\sqrt{2r}}{2\sqrt{s}}\big\|\mathcal{P}_{\mathcal {T}}(\Delta^{\!k})\big\|_F
  \le \frac{\gamma_{k-1}\sqrt{2r}}{2\sqrt{s}}\big\|\mathcal{P}_{\mathcal {M}^k}(\Delta^{\!k})\big\|_F,\nonumber
  \end{align}
  where the second inequality is using the fact that $ab\le (a+b)^2/4$ for $a,b\in\mathbb{R}$.
  The result then follows by noting that
  $\|\Delta^{\!k}\|_F^2=\|\mathcal{P}_{\mathcal {M}^k}(\Delta^{\!k})\|_F^2
  +\|\mathcal{P}_{(\mathcal {M}^k)^\perp}(\Delta^{\!k})\|_F^2$.
  \end{proof}

  \medskip
  \noindent
  {\bf Proof of Proposition \ref{prop1-sec41}:}
  By the definition of $\gamma_{k-1}$ and $\gamma_{k-1}\!\in[0,1/c)$, it is clear that
  $\|\mathcal{P}_{\mathcal {T}^{\perp}}(W^{k-1})\|<1$. From Assumption \ref{assump}
  and Lemma \ref{lemma1-reig} of Appendix B, it follows that
  \begin{equation}\label{ineq-pi}
    \frac{\pi(2r+s,s)\gamma_{k-1}}{s}\le \frac{c_{k-1}}{\sqrt{2r}}\ \ {\rm with}\ \
    c_{k-1}=c\gamma_{k-1}<1.
  \end{equation}
  Applying Lemma \ref{lemma2-reig} of Appendix B with $\mathcal{L}=\mathcal{T}^{\perp}, \mathcal{J}_1=\mathcal{H}^k,
  \mathcal{G}=\mathcal{M}^k, H=\mathcal{P}_{\mathcal{M}^k}(\Delta^{\!k})$ and $G=\Delta^{\!k}$ and noting that
  $\mathcal{P}_{\mathcal{M}^k}(G-H)=0$ since $G-H=\mathcal{P}_{(\mathcal{M}^k)^{\perp}}(\Delta^{\!k})$,
  we have that
  \begin{align}\label{temp-ineq54}
   &\max\big(0,\langle \mathcal{P}_{\mathcal{M}^k}(\Delta^{\!k}),\mathcal{A}^*\mathcal{A}(\Delta^{\!k})\rangle\big)\nonumber\\
   &\ge \vartheta_{-}(2r\!+\!s)\Big(\big\|\mathcal{P}_{\mathcal{M}^k}(\Delta^{\!k})\big\|_F
    -\frac{\pi(2r\!+\!s,s)}{s}\big\|\mathcal{P}_{\mathcal{T}^{\perp}}(\Delta^{\!k})\big\|_*\Big)
    \big\|\mathcal{P}_{\mathcal{M}^k}(\Delta^{\!k})\big\|_F\nonumber\\
  % &\ge \vartheta_{-}(2r\!+\!s)\Big(\|\mathcal{P}_{\mathcal{K}}(\Delta^{\!k})\|_F -\frac{\pi(2r\!+\!s,s)\gamma_{k-1}}{s}\big\|\mathcal{P}_{\mathcal {G}^{\perp}}(\Delta^{\!k})\big\|_*\Big)\|\mathcal{P}_{\mathcal {K}}(\Delta^{\!k})\|_F\nonumber\\
   &\ge \vartheta_{-}(2r\!+\!s)\Big(\|\mathcal{P}_{\mathcal{M}^k}(\Delta^{\!k})\|_F
   -c_{k-1}\big\|\mathcal{P}_{\mathcal {T}}(\Delta^{\!k})\big\|_F\Big)\|\mathcal{P}_{\mathcal{M}^k}(\Delta^{\!k})\|_F\nonumber\\
   &\ge \vartheta_{-}(2r\!+\!s)(1 - c_{k-1})\|\mathcal{P}_{\mathcal{M}^k}(\Delta^{\!k})\|_F^2\ge 0,
  \end{align}
  where the second inequality is due to Lemma \ref{lemma1-sec41} and equation \eqref{ineq-pi},
  and the last one is due to
  $\|\mathcal{P}_{\mathcal {T}}(\Delta^{\!k})\|_F\le\|\mathcal{P}_{\mathcal{M}^k}(\Delta^{\!k})\|_F$.
  In addition, by the definition of $\vartheta_{+}(\cdot)$, it holds that
  \begin{align}
    \max\big(0,\langle \mathcal{P}_{\mathcal{M}^k}(\Delta^{\!k}),\mathcal{A}^*\mathcal{A}\Delta^{\!k}\rangle\big)
     \le\|\mathcal {A}(\mathcal{P}_{\mathcal{M}^k}(\Delta^{\!k}))\|\|\mathcal {A}\Delta^{\!k}\|
     \le 2\delta\sqrt{\vartheta_{+}(2r\!+\!s)}\|\mathcal{P}_{\mathcal{M}^k}(\Delta^{\!k})\|_F.\nonumber
  \end{align}
  Together with \eqref{temp-ineq54}, we obtain that
  \(
   \|\mathcal{P}_{\mathcal{M}^k}(\Delta^{\!k})\|_F
   \le\frac{2\delta\sqrt{\vartheta_{+}(2r+s)}}{(1-c_{k-1})\vartheta_{-}(2r+s)}.
  \)
  The first inequality in \eqref{noisyless-bound} then follows by Lemma \ref{lemma2-sec41}.
  While from Lemma \ref{lemma1-sec41} it follows that
  \[
    \big\|\mathcal{P}_{\mathcal {T}^{\perp}}(X^{k})\big\|_*
    \!\le \!\sqrt{2r}\gamma_{k-1} \big\|\mathcal{P}_{\mathcal {T}}(X^{k})\big\|_F
    \!\le \!\sqrt{2r}\gamma_{k-1} \big\|\mathcal{P}_{\mathcal{M}^k}(X^{k})\big\|_F
    \!\le \frac{2\delta\sqrt{2r}\gamma_{k-1}\sqrt{\vartheta_{+}(2r\!+\!s)}}{(1-c_{k-1})\vartheta_{-}(2r\!+\!s)}.
  \]
  This implies the second inequality in \eqref{noisyless-bound}.
  Thus, we complete the proof. \hfill$\Box$\medskip

  \medskip
  \noindent
  {\bf Proof of Theorem \ref{theorem1-sec41}:} By the definition of $\widetilde{\gamma}_k$
  and Remark \ref{remark-bound}(b), it suffices to prove that
  \begin{subnumcases}{}\label{wabk-order}
   0\le\widetilde{a}_{k}\le\widetilde{a}_{k-1}\le\cdots\le \widetilde{a}_{1}\le \widetilde{b}_1
   \le \cdots\le \widetilde{b}_{k-1}\le\widetilde{b}_{k}\le 1,\\
   0\le\widetilde{\beta}_{k}<\widetilde{\beta}_{k-1}<\cdots<\widetilde{\beta}_{1}<1\ \ {\rm and}\ \
   0\le\gamma_{k}\le \widetilde{\gamma}_{k}<\widetilde{\gamma}_{k-1}<\cdots<\widetilde{\gamma}_{0}.
   \label{betak-order}
 \end{subnumcases}
  By \eqref{dpsi-cg1} and the definitions of $\widetilde{a}_k$ and $\widetilde{b}_k$,
  we have $\{\widetilde{a}_k\}_{k\ge 1}\subseteq[0,1]$ and $\{\widetilde{b}_k\}_{k\ge 1}\subseteq[0,1]$.
  We next establish the monotone relation in \eqref{wabk-order}-\eqref{betak-order} by induction on $k$.
  Let $X^1$ have the SVD as $U^{1}[{\rm Diag}(\sigma(X^1))\ \ 0](V^{1})^{\mathbb{T}}$ where
  $U^{1}=[U_1^{1}\ \ U_2^{1}]\in\mathbb{O}^{n_1}$ and $V^{1}=[V_1^{1}\ \ V_2^{1}]\in\mathbb{O}^{n_2}$
  with $U_1^{1}\in\mathbb{O}^{n_1\times r}$ and $V_1^{1}\in\mathbb{O}^{n_2\times r}$. Then
 \(
    W^1=U^1\big[{\rm Diag}\big(w_1^{1},w_2^{1},\ldots,w_{n_1}^{1})\ \ 0\big](V^1)^{\mathbb{T}}
 \)
 with
 \(
    1\ge w_1^1\ge w_2^1\ge \cdots\ge w_{n_1}^1\ge 0.
 \)
  Since $\gamma_0=1/\sqrt{2}$, the assumption of Proposition \ref{prop1-sec41} holds.
  Then, $\|X^1\!-\!\overline{X}\|_F \le\Xi(\gamma_0)=\Xi(\widetilde{\gamma}_0)$.
  From \cite[Theorem 3.3.16]{HJ91}, it follows that
  \begin{equation}
    \sigma_i(X^1)\ge \sigma_r(\overline{X})-\Xi(\widetilde{\gamma}_0),\ i=1,\ldots,r \ \
    \textrm{and}\ \ \sigma_i(X^1)\leq \Xi(\widetilde{\gamma}_0),\ i=r\!+\!1,\ldots,n_1.\nonumber
  \end{equation}
  Together with the definitions of $\widetilde{a}_1$ and $\widetilde{b}_1$ and inequality \eqref{dpsi-cg2},
  it is easy to deduce that
  \begin{equation}
     w^1_i\geq \widetilde{b}_1,\ i=1,2,\ldots,r\quad\ \textrm{and}\quad \
     t^*\leq w^1_i\leq \widetilde{a}_1,\ i=r+1,\ldots,n_1.\nonumber
  \end{equation}
  This implies that $\widetilde{a}_1\le \widetilde{b}_1$.
  Also, using Lemma \ref{WU1V1} with $\widetilde{X}=\overline{X}$ and $W=W^1$ yields that
  \begin{align}%\label{W-U1V1-sr}
    &\|\mathcal{P}_{\mathcal {T}^\perp}(W^1)\|\leq w^1_{r+1}+(1-w^1_{r+1})\big\|U_1^1(V_1^1)^\mathbb{T}\!-\!\overline{U}_1\overline{V}_1^\mathbb{T}\big\|^2,\nonumber\\
   & \|\mathcal{P}_{\mathcal {T}}(W^{1}\!-\!\overline{U}_1\overline{V}_1^\mathbb{T})\|_F \leq \sqrt{r}(1-\widetilde{b}_1)+(\sqrt{2}\widetilde{a}_1+1)\big\|U_1^1(V_1^1)^\mathbb{T}\!-\!\overline{U}_1\overline{V}_1^\mathbb{T}\big\|.\nonumber
  \end{align}
   Since $\|X^1\!-\!\overline{X}\|_F \le \Xi(\widetilde{\gamma}_0)$,
   applying Lemma \ref{lemma1-A} with $\omega=\sigma_r(\overline{X})/\Xi(\widetilde{\gamma}_0)$,
   $\widetilde{X}=\overline{X}, X=X^1$ and $\eta=\Xi(\widetilde{\gamma}_0)$ we obtain that
   $\|U_1^1(V_1^1)^\mathbb{T}\!-\!\overline{U}_1\overline{V}_1^\mathbb{T}\|\le \widetilde{\beta}_1<1$.
   Then,
   \begin{subnumcases}{}\label{W-U1V1-gr1}
    1\!-\!\|\mathcal{P}_{\mathcal {T}^\perp}(W^{1})\|\geq (1\!-\!\widetilde{a}_1)\big(1\!-\!\widetilde{\beta}_1^2\big),\\
    \|\mathcal{P}_{\mathcal {T}}(W^{1}\!-\!\overline{U}_1\overline{V}_1^\mathbb{T})\|_F \leq \sqrt{r}(1-\widetilde{b}_1)+(\sqrt{2}\widetilde{a}_1+1)\widetilde{\beta}_1.
    \label{W-U1V1-gr2}
   \end{subnumcases}
   Since $\rho_1$ is chosen such that
   \(
    \widetilde{a}_1<\frac{\sqrt{r}(\widetilde{b}_1-\widetilde{\beta}_1^2)-\widetilde{\beta}_1}
    {\sqrt{r}(1-\widetilde{\beta}_1^2)+\sqrt{2}\widetilde{\beta}_1}<1,
   \)
   by the definitions of $\gamma_1$ and $\widetilde{\gamma}_1$, we have that
   $0\leq \gamma_1\leq \widetilde{\gamma}_1<1/\sqrt{2}$.
   Thus, the conclusion holds for $k=1$.

  \medskip

  Now assume that the conclusion holds for $k\le l-1$ with $l\ge 2$.
  We shall show that it holds for $k=l$. Since the conclusion holds for $k=l\!-\!1$,
  we have $\gamma_{l-1}\le\widetilde{\gamma}_{l-1}<1/\sqrt{2}$. This means that the assumption
  of Proposition \ref{prop1-sec41} holds for $k=l$. Consequently,
  \[
    \|X^{l}\!-\!\overline{X}\|_F\le \Xi(\gamma_{l-1})\le \Xi(\widetilde{\gamma}_{l-1}).
  \]
  Let $X^l$ have the SVD given by $X^l\!=\!U^{l}[{\rm Diag}(\sigma(X^{l}))\ \ 0](V^{l})^{\mathbb{T}}$,
  where $U^{l}\!=[U_1^{l}\ \ U_2^{l}]\in\mathbb{O}^{n_1}$ and $V^{l}=[V_1^{l}\ \ V_2^{l}]\in\mathbb{O}^{n_2}$
  with $U_1^{l}\in\mathbb{O}^{n_1\times r}$ and $V_1^{l}\in\mathbb{O}^{n_2\times r}$. Then we have that
  \[
    W^{l}=U^{l}\big[{\rm Diag}\big(w_1^{l},\ldots,w_{n_1}^{l})\ \ 0\big](V^{l})^{\mathbb{T}}
    \ \ {\rm with}\ \ 1\ge w_1^l\ge \cdots\ge w_{n_1}^l\ge 0.
  \]
  From \cite[Theorem 3.3.16]{HJ91},
  $\sigma_i(X^l)\ge \sigma_r(\overline{X})-\Xi(\widetilde{\gamma}_{l-1})$ for $i=1,\ldots,r$
  and $\sigma_i(X^{l})\leq \Xi(\widetilde{\gamma}_{l-1})$ for $i=r\!+\!1,\ldots,n_1$.
  By the definitions of $\widetilde{a}_{l}$ and $\widetilde{b}_{l}$ and
  inequality \eqref{dpsi-cg2},
  \begin{equation}\label{wlt}
     w^{l}_i\geq \widetilde{b}_{l},\ i=1,2,\ldots,r\ \ \textrm{and}\ \
     t^*\leq w^{l}_i\leq \widetilde{a}_{l},\ i=r\!+\!1,\ldots,n_1.
  \end{equation}
  Since the conclusion holds for $k=l\!-\!1$, we have $\Xi(\widetilde{\gamma}_{l-1})<\Xi(\widetilde{\gamma}_0)$,
  and then $\frac{\sigma_r(\overline{X})}{\Xi(\widetilde{\gamma}_{l-1})}>2$.
  Using Lemma \ref{WU1V1} with $\widetilde{X}=\overline{X}$ and $W=W^{l}$ and Lemma \ref{lemma1-A}
  with $\omega=\frac{\sigma_r(\overline{X})}{\Xi(\widetilde{\gamma}_{l-1})},\widetilde{X}=\overline{X}$,
  $X=X^{l}$ and $\eta=\Xi(\widetilde{\gamma}_{l-1})$ and following the same arguments as $k=1$, we obtain that
  \begin{subnumcases}{}\label{Wl-U1V1-gr1}
    1\!-\!\|\mathcal{P}_{\mathcal {T}^\perp}(W^{l})\|\geq (1\!-\!\widetilde{a}_{l})(1\!-\!\widetilde{\beta}_l^2),\\
    \|\mathcal{P}_{\mathcal {T}}(W^{l}\!-\!\overline{U}_1\overline{V}_1^\mathbb{T})\|_F
    \leq \sqrt{r}(1-\widetilde{b}_l)+(\sqrt{2}\widetilde{a}_l+1)\widetilde{\beta}_l.
    \label{Wl-U1V1-gr2}
  \end{subnumcases}
  Notice that $1\leq \mu_{l}\leq \frac{\Xi(\widetilde{\gamma}_{l-2})}{\Xi(\widetilde{\gamma}_{l-1})}$.
  So, $\rho_{l-1}\le\rho_{l}\le\frac{\rho_{l-1}\Xi(\widetilde{\gamma}_{l-2})}{\Xi(\widetilde{\gamma}_{l-1})}$.
  By the definitions of $\widetilde{a}_l$ and $\widetilde{b}_l$ and inequality \eqref{dpsi-cg2},
  $\widetilde{a}_{l}\leq \widetilde{a}_{l-1}$ and $\widetilde{b}_{l}\geq \widetilde{b}_{l-1}$.
  In addition, noting that $\Xi(\widetilde{\gamma}_{l-1})<\Xi(\widetilde{\gamma}_{l-2})$
  since $\widetilde{\gamma}_{l-1}<\widetilde{\gamma}_{l-2}$, we also have
  $\widetilde{\beta}_{l}<\widetilde{\beta}_{l-1}$. Equations \eqref{Wl-U1V1-gr1}
  and \eqref{Wl-U1V1-gr2} and the definitions of $\gamma_l$ and $\widetilde{\gamma}_{l}$
  imply that
 \(
   0\leq \gamma_{l}\leq \widetilde{\gamma}_{l}<\widetilde{\gamma}_{l-1}.
 \)
  Thus, the conclusion holds for $k=l$.  \hfill$\Box$\medskip

  \medskip
  \noindent
  {\bf Proof of Theorem \ref{theorem1-sec42}:}
  Notice that the assumption of Theorem \ref{theorem1-sec41} is satisfied.
  The monotone relation in \eqref{wabk-order} and \eqref{betak-order} holds
  for all $k\ge 2$. Clearly, inequality \eqref{gconverge1} holds for $k=1$.
  Now fix $k\ge 2$. Let $X^{k-1}$ have the SVD as
  $U^{k-1}[{\rm Diag}(\sigma(X^{k-1}))\ \ 0](V^{k-1})^{\mathbb{T}}$, where
  $U^{k-1}=[U_1^{k-1}\ \ U_2^{k-1}]\in\mathbb{O}^{n_1}$ and $V^{k-1}=[V_1^{k-1}\ \ V_2^{k-1}]\in\mathbb{O}^{n_2}$
  with $U_1^{k-1}\in\mathbb{O}^{n_1\times r}$ and $V_1^{k-1}\in\mathbb{O}^{n_2\times r}$.
  Then $W^{k-1}=U^{k-1}[{\rm Diag}\big(w_1^{k-1},w_2^{k-1},\ldots,w_{n_1}^{k-1})\ \ 0] (V^{k-1})^{\mathbb{T}}$
  with $1\ge w_1^{k-1}\ge \cdots\ge w_{n_1}^{k-1}\!\ge 0$. Using the same arguments as those for
  Theorem \ref{theorem1-sec41}, we get
  \begin{align*}
   1\!-\!\|\mathcal{P}_{\mathcal {T}^\perp}(W^{k-1})\|\geq (1\!-\!\widetilde{a}_{k-1})(1\!-\!\widetilde{\beta}_{k-1}^2),
   \qquad\qquad\qquad\qquad\\
    \big\|\mathcal{P}_{\mathcal {T}}(W^{k-1}\!-\!\overline{U}_1\overline{V}_1^\mathbb{T})\big\|_F \leq \sqrt{r}(1\!-\!\widetilde{b}_{k-1})+(1\!+\!\sqrt{2}\widetilde{a}_{k-1})
    \big\|U_1^{k-1}(V_1^{k-1})^\mathbb{T}\!-\!\overline{U}_1\overline{V}_1^\mathbb{T}\big\|.
  \end{align*}
  Also, from \cite[Equation(49)-(51)]{MiaoPS16},
  \(
    \big\|U_1^{k-1}(V_1^{k-1})^\mathbb{T}\!-\!\overline{U}_1\overline{V}_1^\mathbb{T}\big\|_F\le \frac{\|X^{k-1}-\overline{X}\|_F}{\sigma_{r}(\overline{X})\!-\!\sqrt{2}\Xi(\gamma_0)}.
  \)
  Thus, together with the definition of $\gamma_{k-1}$, we immediately obtain that
  \[
    \gamma_{k-1}
    \le \frac{1-\widetilde{b}_{1}}{\sqrt{2}(1-\widetilde{a}_{1})(1\!-\!\widetilde{\beta}_{1}^2)} + \frac{1+\sqrt{2}\widetilde{a}_{1}}{\sqrt{2r}(1-\widetilde{a}_{1})(1\!-\!\widetilde{\beta}_{1}^2)}
    \cdot\frac{\|X^{k-1}-\overline{X}\|_F}{\sigma_{r}(\overline{X})\!-\!\sqrt{2}\Xi(\gamma_0)}.
  \]
  From the first part of Theorem \ref{theorem1-sec41} and the first inequality of \eqref{noisyless-bound},
  it follows that
   \begin{align}\label{pcc1}
    \|X^{k}\!-\!\overline{X}\|_F
    &\le \frac{2\delta\sqrt{\vartheta_{+}(2r\!+\!s)}}{(1\!-\!c\gamma_{k-1})\vartheta_{-}(2r\!+\!s)}
         \Big(1\!+\!\sqrt{\frac{r}{2s}}\gamma_{k-1}\Big)
         =\frac{\Xi(0)}{(1\!-\!c\gamma_{k-1})}
         \Big(1\!+\!\sqrt{\frac{r}{2s}}\gamma_{k-1}\Big)\nonumber\\
    &\le \frac{\Xi(0)}{1\!-\!c\widetilde{\gamma}_1}
        \Big[1+\frac{(1-\widetilde{b}_{1})\sqrt{r}}{2(1\!-\!\widetilde{a}_{1})(1\!-\!\widetilde{\beta}_{1}^2)\sqrt{s}}\Big]
       \!+\!\Big[\frac{\alpha\,\Xi(\gamma_0)}{\sigma_{r}(\overline{X})\!-\!\sqrt{2}\Xi(\gamma_0)}\Big]\|X^{k-1}\!-\!\overline{X}\|_F
   \end{align}
  where the second inequality is using
  $\Xi(\gamma_0)=\frac{\Xi(0)}{1\!-\!c\widetilde{\gamma}_0}\sqrt{\frac{4s+r}{4s}}$.
  Since $\sigma_r(\overline{X})>(\sqrt{2}+\alpha)\Xi(\gamma_0)$ implies
  $0\le\frac{\alpha\,\Xi(\gamma_0)}{\sigma_{r}(\overline{X})\!-\!\sqrt{2}\Xi(\gamma_0)}<1$,
  the desired inequality follows by the recursion \eqref{pcc1}.
   \hfill$\Box$\medskip
%--------------------------------------------------------------------------------------------
  \begin{alemma}\label{lemma-noise}
  If the components $\xi_1,\xi_2,\ldots,\xi_m$ of $\xi$ are independent sub-Gaussians,
  then $\|\xi\|\le \sqrt{m}\sigma$ with probability at least $1-\exp(1-\frac{c_1m}{4})$
  for an absolute constant $c_1>0$.
 \end{alemma}
 \begin{proof}
  Notice that
  \(
    \|\xi\|=\sup_{u\in \mathcal{S}^{m-1}}\langle u,\xi\rangle,
  \)
  where $\mathcal{S}^{m-1}$ denotes the unit sphere in $\mathbb{R}^m$.
  Let $\mathcal{U}:=\{u^1,\ldots,u^{m}\}$ denote $1/2$ covering of $\mathcal{S}^{m-1}$.
  Then, for any $u\in \mathcal{S}^{m-1}$, there exists $\overline{u}\in\mathcal{U}$
  such that $u=\overline{u}+\Delta u$ with $\|\Delta u\|\le 1/2$.
  Consequently,
  \(
    \langle u,\xi\rangle =\langle \overline{u},\xi\rangle
    +\langle\Delta u,\xi\rangle  \le \langle \overline{u},\xi\rangle + \frac{1}{2}\|\xi\|.
  \)
  This, by $\|\xi\|=\sup_{u\in \mathcal{S}^{m-1}}\langle u,\xi\rangle$, implies that
  $\|\xi\|\le 2\langle \overline{u},\xi\rangle=2\sum_{i=1}^m\overline{u}_i\xi_i$.
  By applying the Hoeffding-type inequality (see \cite{Vershynin11}),
  we know that there exists    an absolute constant $c_1$ such that for every $t>0$,
  \[
   \mathbb{P}\left\{\|\xi\|\ge t\right\} \le \mathbb{P}\left\{\big|{\textstyle\sum_{i=1}^m} \overline{u}_i\xi_i\big|\ge t/2\right\}
    \le \exp\big(1-\!c_1t^2/(4\sigma^2)\big).
  \]
  Taking $t=\sqrt{m}\sigma$ yields the desired result.
  \end{proof}
 \end{document}